\documentclass[a4paper, 10pt, twoside, notitlepage]{amsart}

\usepackage[utf8]{inputenc}
\usepackage{color}
\usepackage{amsmath} 
\usepackage{amssymb} 
\usepackage{amsthm}
\usepackage{geometry}
\usepackage{graphicx}
\graphicspath{{figures/}}
\usepackage{esint}
\usepackage[colorlinks=true,linkcolor=blue]{hyperref}

\theoremstyle{plain}
\newtheorem{thm}{Theorem}
\newtheorem{prop}{Proposition}[section]
\newtheorem{lem}[prop]{Lemma}
\newtheorem{cor}[prop]{Corollary}

\newtheorem{rmk}[prop]{Remark}
\newtheorem{claim}[prop]{Claim}

\newcommand {\R} {\mathbb{R}} 
 \newcommand {\N} {\mathbb{N}}
 
\newcommand {\p} {\partial}

\newcommand {\supp} {{\rm{supp}}}

\newcommand {\diag} {\text{diag}}
\newcommand {\rank} {\rm{rank}}
\newcommand {\conv} {\text{conv}}

\DeclareMathOperator {\dist} {dist}

\DeclareMathOperator{\F} {\mathcal{F}}

\title[On Scaling Laws for Multi-Well Nucleation Problems]{On Scaling Laws for Multi-Well Nucleation Problems without Gauge Invariances}
\author{Angkana Rüland}
\address{Institut f\"ur Angewandte Mathematik, Ruprecht-Karls-Universit\"at Heidelberg, Im Neuenheimer Feld 205, 69120 Heidelberg, Germany}
\email{Angkana.Rueland@uni-heidelberg.de}
\author{Antonio Tribuzio}
\address{Institut f\"ur Angewandte Mathematik, Ruprecht-Karls-Universit\"at Heidelberg, Im Neuenheimer Feld 205, 69120 Heidelberg, Germany}
\email{Antonio.Tribuzio@uni-heidelberg.de}

\begin{document}

\begin{abstract}
In this article we study scaling laws for simplified multi-well nucleation problems without gauge invariances which are motivated by models for shape-memory alloys. Seeking to explore the role of the order of lamination on the energy scaling for nucleation processes, we provide scaling laws for various model problems in two and three dimensions. In particular, we discuss (optimal) scaling results in the volume and the singular perturbation parameter for settings in which the surrounding parent phase is in the first, the second and the third order lamination convex hull of the wells of the nucleating phase. Furthermore, we provide a corresponding result for the setting of an infinite order laminate which arises in the context of the Tartar square. In particular, our results provide isoperimetric estimates in situations in which strong nonlocal anisotropies are present.
\end{abstract}
\maketitle

\section{Introduction}

Motivated by nucleation problems in shape-memory alloys, in this article, we study scaling laws for simplified, highly non-convex multi-well model problems with a prescribed volume constraint and with surface energy regularizations. Adopting a variational point of view and considering these results in the context of phase-transforming systems, we interpret these scaling laws as nucleation barriers for the nucleation of one phase within a matrix of the other phase. Typical systems which we have in mind are, for instance, the description of martensitic inclusions within an austenite matrix in the context of phase-transformations in shape-memory-alloys \cite{Bhat, M1}. Mathematically these questions correspond to isoperimetric problems in which there is a competition between a perimeter contribution and strong nonlocal anisotropies which are dictated by the multi-well energies.

\subsection{The models and the general setting}

Let us describe the models which we are considering in the following sections in more detail:
We study an inclusion of a phase with several variants -- in the shape-memory context this would correspond to the ``martensite'' phase -- inside a parent phase, the ``austenite'' phase. Following \cite{BJ92,B4} we adopt a variational perspective and consider elastic energies of the form
\begin{equation}
\label{eq:elast1}
E_{el}(u)=\int_{\R^n}\dist^2(\nabla u,K_0)dx,
\quad
K_0=K\cup\{\textbf{0}\} \subset \R^{n\times n}
\end{equation}
with $n\ge 2$ and $u\in H^1(\R^n;\R^n)$.
Here, physically, $u:\R^n \rightarrow \R^n$ models the deformation of the material and the set $K_0$ corresponds to the stress-free states at the critical temperature. We use $\textbf{0} \in \R^{n\times n}$ to denote the zero matrix, which models the austenite phase. Seeking to provide optimal scaling laws, we simplify the problem and do not include the typical gauge invariances arising from the physical requirement of frame-indifference, but study quantitative and discrete versions of $m$-well problems as proposed in \cite{B3}. Here, a main objective is the investigation of the role of the order of lamination of the parent phase with respect to the nucleating phase. To this end, we focus on sets $K\subset\R^{n\times n}$ of diagonal matrices. Moreover, we concentrate on situations in which a suitable lamination or rank-one convex hull of $K$ contains the zero matrix.

Following \cite{Bhat,CO,CO1}, we express the elastic energies with the help of phase indicators which allows to decouple the gradient and the phase indicator (i.e. the projection of $\nabla u$ onto $K_0$). This leads to elastic energies of the following form
\begin{equation}\label{eq:el-chi}
E_{el}(\chi):=\inf_{u\in H^1(\R^n;\R^n)}E_{el}(u,\chi),
\quad \text{where} \quad
E_{el}(u,\chi):=\int_{\R^n}|\nabla u-\chi|^2dx
\end{equation}
with $\chi\in BV(\R^n;K_0)$ and $\chi=\diag(\chi_{1,1},\dots,\chi_{n,n})$.

Seeking to study the nucleation behaviour of a minority phase in a majority phase and to also include surface energies into the model, we further introduce the length scale $\epsilon>0$ by adding to the elastic energy $E_{el}$ a singular higher-order term of the type
$$
E_{surf}(\chi):=|D\chi|(\R^n),
$$
where $|D\chi|(\R^n)$ denotes the total variation semi-norm of $\chi$.

This leads to a singular perturbation problem consisting of a combination of elastic and surface energies
\begin{equation}\label{eq:en-tot}
E_\epsilon(u,\chi) := E_{el}(u,\chi)+\epsilon E_{surf}(\chi)
\quad \text{and} \quad
E_\epsilon(\chi) := E_{el}(\chi)+\epsilon E_{surf}(\chi).
\end{equation}

In this setting, it is our main objective to investigate the scaling behaviour of the minimal energy depending on the volume of the inclusion, thus providing (almost) matching upper and lower scaling bounds for the quantity
\begin{equation}\label{eq:en-vol}
E_\epsilon(V):=\inf\Big\{E_\epsilon(\chi) \,:\, |\supp(\chi)|=V\Big\}.
\end{equation}
A key emphasis here will be on the role of the order of lamination of the parent phase with respect to the nucleating phase which we aim to investigate in two and three dimensions for model problems.

\subsection{The main results}

We study the scaling behaviour in the described isoperimetric problems depending on the order of lamination of the parent phase with respect to the minority phase. In what follows, we thus consider four model problems and determine the corresponding scaling laws for the nucleation of a martensitic nucleus in the austenite parent phase in the respective settings.

\subsubsection{Laminates of first order}
\label{sec:intro_first_order}
As the most basic example, we begin by studying the situation in which the parent phase is rank-one connected with the nucleating phase. Here we distinguish two settings: In the first case, only two matrices are exactly stress-free (the martensite and the austenite), while in the second case three matrices -- two variants of martensite in the inclusion and one in the parent phase -- are exactly stress free.

We begin by considering the case of two stress-free states, for which $K=\{A\}$ for some $A\in \diag(n,\R)$ with $\rank(A) = 1$.  A similar, more complex setting including a linear gauge symmetry (in the linearized theory of elasticity) had been studied in \cite{KK} where optimal scaling bounds had been deduced. Relying on the argument from \cite{KK}, also in our setting, we recover the same scaling law behaviour in this situation:

\begin{thm}[Theorem 2.1 \cite{KK}]
\label{thm:KK}
Let $E_{\epsilon}(V)$ be as in \eqref{eq:en-vol} and let $K=\{A\}$ for some $A\in\diag(n,\R)$ with $\rank(A)=1$. Then there exist two positive constants $C_2>C_1>0$ depending on $K$ and $n$ such that for every $V>0$ and  every $\epsilon>0$ there holds
$$
C_1 r_\epsilon(V) \le E_\epsilon(V) \le C_2 r_\epsilon(V),
\quad \text{where} \quad
r_\epsilon(V)=\begin{cases}
\epsilon V^\frac{n-1}{n} & V\le\epsilon^n, \\
\epsilon^\frac{n}{2n-1}V^\frac{2n-2}{2n-1} & V>\epsilon^n.
\end{cases}
$$
\end{thm}

We refer to the forthcoming work \cite{AKKR22} for such a result within the geometrically nonlinear theory of elasticity having full $SO(2)$ invariance.

Turning to the setting of $\# K = 2$ and imposing that $\textbf{0}$ is rank-one connected with the variants of martensite, we may without loss of generality assume that 
\begin{align}
\label{eq:two}
K=\{A,B\},
\quad \text{with} \quad
A= \begin{pmatrix} -\lambda & 0 \\ 0 & 0\end{pmatrix},\,
B= \begin{pmatrix} 1-\lambda & 0 \\ 0 & 0 \end{pmatrix}
\end{align}
for some $\lambda\in(0,1)$ fixed, with ``austenite" given by the zero matrix. In this case $\textbf{0} \in K^{(1)}$, which denotes the first-order lamination-convex hull (and coincides with $\conv(K)$ in this case). In contrast to the previous case in which $\#K =1$, in the nucleation process it is now possible for the nucleating phase to form an internal microstructure lowering the elastic energy. This leads to an overall improved energy scaling behaviour in this setting for large volume fractions:

\begin{thm}
\label{thm:2wells}
Let $E_{\epsilon}(V)$ be as in \eqref{eq:en-vol} and let $K$ be as in \eqref{eq:two}. 
Then, there exist two positive constants $C_2>C_1>0$ depending on $K$ such that for every $V>0$ and for every $\epsilon>0$ there holds
\begin{align*}
C_1 r_\epsilon(V) \le E_\epsilon(V) \le C_2 r_\epsilon(V),
\quad \text{where} \quad
r_\epsilon(V)=\begin{cases}
\epsilon V^\frac{1}{2} & V\le\epsilon^2, \\
\epsilon^\frac{4}{5}V^\frac{3}{5} & V>\epsilon^2.
\end{cases}
\end{align*}
\end{thm}

Moreover, it is possible to generalize this to arbitrary dimensions, picking up a dimensional dependence:

\begin{cor}
\label{cor:lamination_1}
Let $E_{\epsilon}(V)$ be as in \eqref{eq:en-vol} and let $K=\{-\lambda e_1\otimes e_1, (1-\lambda)e_1\otimes e_1\}$ with $\lambda \in (0,1)$ and $e_1$ denoting the first, normalized canonical basis vector.
Then there exist two positive constants $C_2>C_1>0$ depending on $K$ and $n$ such that for every $V>0$ and for every $\epsilon>0$ there holds
\begin{align*}
C_1 r_\epsilon(V) \le E_\epsilon(V) \le C_2 r_\epsilon(V),
\quad \text{where} \quad
r_\epsilon(V)=\begin{cases}
\epsilon V^\frac{n-1}{n} & V\le\epsilon^n, \\
\epsilon^\frac{2n}{3n-1}V^\frac{3n-3}{3n-1} & V>\epsilon^n.
\end{cases}
\end{align*}
\end{cor}

\subsubsection{A second order laminate in two and three dimensions}
\label{sec:intro_second_order}

Considering nucleation settings in which the parent phase is a \emph{second order laminate}, the elastic energy plays a stronger role in the resulting scaling law. Indeed, we expect that more microstructure is necessary in order to obey self-accommodation of the nucleus on the one hand and to ensure compatibility of the phases on the other hand. Seeking to study these effects, we consider the following set
$K=\{A_1,A_2,A_3,A_4\}$ with
\begin{align}
\label{eq:4_wells}
A_1= \begin{pmatrix} -1 & 0 \\ 0 & -2\end{pmatrix},\,
A_2= \begin{pmatrix} -1 & 0 \\ 0 & 1 \end{pmatrix},\,
A_3= \begin{pmatrix} 1 & 0 \\ 0 & 2 \end{pmatrix},\,
A_4= \begin{pmatrix} 1 & 0 \\ 0 & -1 \end{pmatrix}.
\end{align}
In this setting, we have first and second order laminates given by the following formulas:
$K^{(1)}:= \conv\{A_1,A_2\} \cup\conv\{A_3,A_4\}$ and 
\begin{align*}
K^{(2)}\setminus K^{(1)} = \left\{
\begin{pmatrix}
\mu & 0 \\ 0 & \nu
\end{pmatrix} \,:\, |\mu|<1,|\nu|\le 1\right\},
\end{align*}
see Figure \ref{fig:4wells}.

\begin{figure}[thb]
\includegraphics{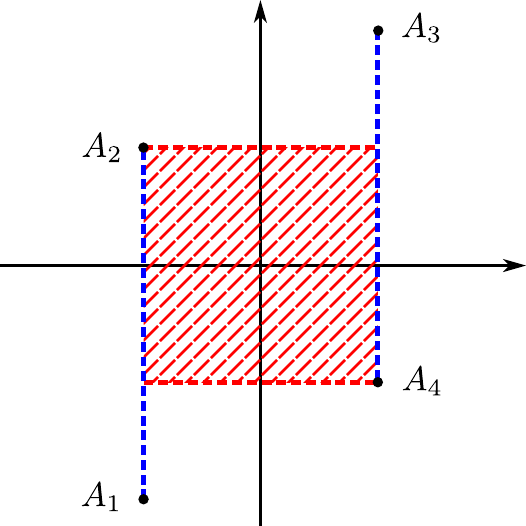}
\caption{The matrix set from Section \ref{sec:intro_second_order}: In blue the set $K^{(1)}\setminus K$, in red $K^{(2)}\setminus K^{(1)}$.}
\label{fig:4wells}
\end{figure}

In particular, the zero matrix belongs to the lamination-convex hull $K^{lc}$ and can be obtained as
$$
\frac{1}{2}\Big(\frac{1}{3}A_1+\frac{2}{3}A_2\Big)+\frac{1}{2}\Big(\frac{1}{3}A_3+\frac{2}{3}A_4\Big)=\textbf{0},
$$
that is the zero matrix is a second order laminate of $K$.

In this setting, it turns out that indeed, the nucleation process is more expensive for large nucleation cores than in the setting of Theorem \ref{thm:2wells}:

\begin{thm}
\label{thm:second_intro}
Let $E_{\epsilon}(V)$ be as in \eqref{eq:en-vol} and let $K$ be given by the matrices in \eqref{eq:4_wells}.
Then, there exist two positive constants $C_2>C_1>0$ depending on $K$ such that for every $V>0$ and for every $\epsilon>0$ there holds
\begin{align*}
C_1 r_\epsilon(V) \le E_\epsilon(V) \le C_2 r_\epsilon(V),
\quad \text{where} \quad
r_\epsilon(V)=\begin{cases}
\epsilon V^\frac{1}{2} & \text{if } V\le\epsilon^2, \\
\epsilon^\frac{4}{7}V^\frac{5}{7} & \text{if } V>\epsilon^2.
\end{cases}
\end{align*}
\end{thm}

We remark that also for this setting, higher dimensional analoga can be obtained. Indeed, we refer to Section \ref{sec:3D} (Proposition \ref{prop:3D_analogue}) in which a lower scaling bound for a second order laminate parent phase in three dimensions is investigated. For this simple model problem we in particular recover the lower scaling bound which had earlier been derived for the cubic-to-tetragonal phase transition in the geometrically linearized theory of elasticity in \cite{KKO13} (with gauge group $Skew(3)$). As a consequence, in spite of our substantial simplification in ignoring gauges, these models may capture some of the mathematical features of physically more realistic models and can mathematically thus be regarded as interesting, simpler substitutes for these which may allow for a more detailed analysis.

\subsubsection{A three-dimensional third order laminate}
\label{intro:third-order}

In addition to the two-dimensional settings, we also consider a three-dimensional problem for which the zero matrix is a third order laminate. To this end, we consider the set of wells given by
\begin{align}
\label{eq:three_dim_3rd}
\begin{split}
K &= \left\{ 
\begin{pmatrix}
2 & 0 & 0 \\ 
0 & -2 & 0\\
0 & 0 & 1
\end{pmatrix},
\begin{pmatrix}
-2 & 0 & 0 \\ 
0 & -2 & 0\\
0 & 0 & 1
\end{pmatrix},
\begin{pmatrix}
-3 & 0 & 0 \\ 
0 & 2 & 0\\
0 & 0 & 1
\end{pmatrix},
\begin{pmatrix}
3 & 0 & 0 \\ 
0 & 2 & 0\\
0 & 0 & 1
\end{pmatrix},\right.\\
& \quad \left.\begin{pmatrix}
1 & 0 & 0 \\ 
0 & -1 & 0\\
0 & 0 & -1
\end{pmatrix},
\begin{pmatrix}
-1 & 0 & 0 \\ 
0 & -1 & 0\\
0 & 0 & -1
\end{pmatrix},
\begin{pmatrix}
-4 & 0 & 0 \\ 
0 & 1 & 0\\
0 & 0 & -1
\end{pmatrix},
\begin{pmatrix}
4 & 0 & 0 \\ 
0 & 1 & 0\\
0 & 0 & -1
\end{pmatrix}
\right\},
\end{split}
\end{align}
see also Figure \ref{fig:4wells-3D}.
In terms of the characteristic functions the associated stress-free inclusion turns into
\begin{align*}
\nabla u \in \begin{pmatrix} \chi_{1,1} & 0 & 0\\
0 & \chi_{2,2} & 0 \\
0 & 0 & \chi_{3,3}
 \end{pmatrix},
\end{align*}
with
\begin{align*}
\chi_{1,1} & =2 \chi_1 - 2\chi_2 - 3\chi_3+ 3 \chi_4 + \chi_5- \chi_6-4\chi_7 + 4\chi_8,\\
\chi_{2,2} &= -2\chi_1 - 2\chi_2 + 2\chi_3+ 2\chi_4 - \chi_5-\chi_6+\chi_7 +\chi_8,\\
\chi_{3,3} &= \chi_1 + \chi_2 + \chi_3 + \chi_4-\chi_5-\chi_6 - \chi_7 - \chi_8.
\end{align*}

\begin{figure}[t]
\includegraphics{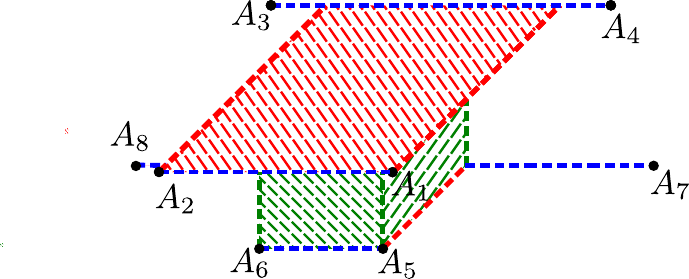}
\caption{The matrix set from Section \ref{intro:third-order}: In blue the set $K^{(1)}\setminus K$, in red $K^{(2)}\setminus K^{(1)}$, in green $K^{(3)}\setminus K^{(2)}$.}
\label{fig:4wells-3D}
\end{figure}

This now contains boundary data with lamination orders up to order three. In particular, the zero matrix has lamination order three. In this setting, we consider lower scaling bounds and show that again these are determined by the order of lamination:

\begin{prop}[Lower bounds for third order laminates in three dimensions]
\label{prop:lower_3}
Let $E_{\epsilon}(V)$ be as in \eqref{eq:en-vol} and let $K$ be as \eqref{eq:three_dim_3rd}.
Then, there exists a positive constant $C>0$ depending on $K$ such that for every $V>0$ and for every $\epsilon>0$ it holds
\begin{align*}
E_{\epsilon}(V) \geq C r_{\epsilon}(V), \mbox{ where } 
r_{\epsilon}(V)= 
\left\{
\begin{array}{ll} 
\epsilon V^{\frac{2}{3}}, & \ V \leq \epsilon^{3},\\
\epsilon^{\frac{3}{7}} V^{\frac{6}{7}}, & \ V > \epsilon^{3}.
\end{array} \right.
\end{align*}
\end{prop}

We expect that matching upper bounds could be proved using the three-dimensional construction from \cite[Proposition 6.3]{RT21}. As this however is technically rather involved, we do not provide the matching upper bounds here but postpone this to future work.

\subsubsection{A laminate of infinite order}
\label{sec:intro_tartar}

Last but not least, we study a setting which is ``almost rigid'' in that the zero matrix is an infinite order laminate of the nucleating phases. Here, the set $K$ consists of the matrices forming the Tartar square  \cite{Sch75, AH86,NM91,CT93,T93}, i.e. they are given by $K=\{A_1,A_2,A_3,A_4\} \subset \R^{2\times 2}$ with
\begin{align}
\label{eq:T4_wells}
A_1= \begin{pmatrix} -1 & 0 \\ 0 & -3\end{pmatrix},\,
A_2= \begin{pmatrix} -3 & 0 \\ 0 & 1 \end{pmatrix},\,
A_3= \begin{pmatrix} 1 & 0 \\ 0 & 3 \end{pmatrix},\,
A_4= \begin{pmatrix} 3 & 0 \\ 0 & -1 \end{pmatrix}.
\end{align}

In this setting, very complicated microstructure has to emerge in order to ensure self-accommo-dation and compatibility. Hence, the infinite order of lamination is reflected in a very rigid, high energy scaling law behaviour:

\begin{thm}
\label{thm:Tartar}
Let $E_{\epsilon}(V)$ be as in \eqref{eq:en-vol} and let $K$ consist of the matrices in \eqref{eq:T4_wells}.
Then, there exist four positive constants $C^{(1)},C^{(2)}>0$, $C_2>C_1>0$ depending on $K$ such that for every $V>0$ and for every $\epsilon>0$ there holds
\begin{align*}
C_1 r_\epsilon^{(1)}(V) \le E_\epsilon(V) \le C_2 r_\epsilon^{(2)}(V),
\quad \text{where} \quad
r_\epsilon^{(j)}(V) =\begin{cases}
\epsilon V^\frac{1}{2} & \text{if } V\le\epsilon^2, \\
V\exp\Big(-C^{(j)}\log\Big(\frac{V}{\epsilon^2}\Big)^\frac{1}{2}\Big) & \text{if } V>\epsilon^2.
\end{cases}
\end{align*}
\end{thm}

The results from Sections \ref{sec:intro_first_order}-\ref{sec:intro_tartar} illustrate the relevance of the order of lamination of the parent phase with respect to the nucleating phase. We further conjecture that with ideas as the ones outlined in the previous and the following sections, it is possible to produce wells $K\subset \diag(n,\R)$ with nucleation scaling behaviour of the order 
\begin{align*}
C_1 r_\epsilon(V) \le E_\epsilon(V) \le C_2 r_\epsilon(V),
\quad \text{where} \quad
r_\epsilon(V) =\begin{cases}
\epsilon V^\frac{n-1}{n} & \text{if } V\le\epsilon^n, \\
\epsilon^{\frac{2n}{n^2+2n-1}} V^{\frac{n^2+2n-3}{n^2+2n-1}} & \text{if } V>\epsilon^n,
\end{cases}
\end{align*}
 in $n$-dimensional situations for constants $0<C_1\leq C_2$. In particular, this ``interpolates'' between the rather low energy $2+1$ well case in two dimensions and the very rigid, energetically expensive Tartar setting.
In addition to this, let us however also caution that the lamination order of the parent phase with respect to the nucleating phase certainly is \emph{not} the only factor determining the scaling behaviour in nucleation problems. Indeed, considering for instance two-dimensional stair-case laminates, it is possible to create arbitrary high orders of lamination. In this situation it is however not expected that this is necessarily reflected in the nucleation scaling law. We postpone more detailed results on this to future work.

\subsection{Relation to the literature}

The results derived in this article mathematically fall in the class of isoperimetric inequalities \cite{M12} in which anisotropies are present. Due to the presence of the strong anisotropies, for large volumes, balls are in general no longer minimizers of these isoperimetric problems, but interesting microstructure emerges as a competition between the anisotropic nonlocal and the surface energies.
These questions arise naturally in nucleation processes in materials science. In the context of shape-memory alloys, these have, for instance, been studied for the incompatible two-well problem \cite{CM99}, for the cubic-to-tetragonal phase transformation in the geometrically linearized theory of elasticity \cite{KKO13}, for the geometrically linearized two-well problem \cite{KK} and the boundary nucleation for the cubic-to-tetragonal phase transformation \cite{BG15}. The location of nucleation was also studied in \cite{BK16}.
In the absence of self-accommodation bounds have been deduced for the cubic-to-tetragonal phase transformation in \cite{KO19}. Moreover, special constructions and behaviour is known in highly symmetric situations in two dimensions \cite{CKZ17,CDPRZZ20}.

Many of these results and ideas are closely related to singular perturbation problems for shape-memory materials with prescribed boundary conditions (see, for instance, \cite{KM1, KM2,CM99,C99, CDMZ20, CO, CO1, C1, CC14, CC15,CZ16,CDZ17, L01,L06, Rue16b,RT22, RT21}) and the use of scaling as a selection mechanism for wild microstructure \cite{RTZ19, RZZ18}. 
Although in spirit similar to these singular perturbation problems, due to the \emph{flexibility} of the \emph{domain geometry}, in the nucleation setting the problem is less constrained and has more freedom to relax, in part leading to interesting new behaviour.

We emphasize that related questions and results have also been considered in other physical systems such as in compliance minimization \cite{KW14, KW16}, micromagnetism \cite{CKO99} or in models motivated by Coulombic interactions with radially symmetric nonlocal contributions \cite{KM13,KM14}.

\subsection{Outline of the article}
The remainder of the article is structured as follows: After briefly recalling a number of auxiliary results in Section \ref{sec:prelim}, we deal with the scaling laws from Section \ref{sec:intro_first_order}-\ref{sec:intro_tartar} in individual sections splitting the proofs into lower bound estimates and upper bound constructions. Thus, in Section \ref{sec:1+1} we begin by recalling the result from \cite{KK} and illustrating how the method from \cite{KK} yields the behaviour of the 1+1 well case. In Section \ref{sec:2+1} we invoke Fourier tools to deduce improved lower bounds and complement these with matching upper bounds leading to the proof of Theorem \ref{thm:2wells}. Building on these ideas, in Section \ref{sec:second_lam} we discuss second order laminates in two and three dimensions. In three dimensions we in particular recover the lower scaling bound from \cite{KKO13} in our model problem (see Proposition \ref{prop:3D_analogue}). In Section \ref{sec:third_order} we then show that the lower bound estimates are robust and can also be applied to deduce bounds for third order laminates in three dimension which results in the proof of Proposition \ref{prop:lower_3}. Finally, in Section \ref{sec:tartar} we deduce scaling bounds for infinite order laminates for the Tartar square.

\section{Preliminary results}
\label{sec:prelim}

In this section we collect some intermediate results that will be used in the following sections.
Although we mainly treat two- and three-dimensional problems, presenting the results of this section in their general version requires no additional effort.

Here and in what follows, when writing $a\sim b$ we mean that $c^{-1}a\le b\le ca$ for some constant $c>0$ independent of $\epsilon$ and $V$.
Analogously we will write $a\lesssim b$ and $a\gtrsim b$ meaning $a\le cb$ and $a\ge cb$, respectively.

\subsection{Normalization}
\label{rmk:normal}

In considering and estimating our energies which a priori depend on the two parameters $\epsilon$ and $V$, we can always reduce ourselves to a one-parameter problem by a normalization argument.
Let $E_\epsilon$ be as in \eqref{eq:en-tot} and let $u\in H^1(\R^n;\R^n)$, $\chi\in BV(\R^n;K_0)$ be given.
By setting $u_\epsilon(x):=\epsilon^{-1} u(\epsilon x)$, $\chi_\epsilon(x):=\chi(\epsilon x)$, we obtain
$$
E_\epsilon(u,\chi)=\epsilon^n E_1(u_\epsilon,\chi_\epsilon)
\quad \text{and} \quad
E_\epsilon(\chi)=\epsilon^n E_1(\chi_\epsilon).
$$
Indeed, the total variation scales like $\epsilon^{n-1}$, namely $|D\chi|(\R^n)=\epsilon^{n-1}|D\chi_\epsilon|(\R^n)$.
Moreover, $E_{el}(u,\chi)=\epsilon^n E_{el}(u_\epsilon,\chi_\epsilon)$, and therefore $E_{el}(\chi)=\epsilon^n E_{el}(\chi_\epsilon)$, follows by a standard change of variables.

The previous observations justify the following definitions
\begin{equation}\label{eq:en-scaled}
\hat E(u,\chi):= E_1(u,\chi)
\quad \text{and} \quad
\hat E(\chi):= E_1(\chi).
\end{equation}
As a consequence, in what follows we study the scaling behaviour of
\begin{equation}\label{eq:en-vol-scaled}
\hat E(V):=\inf\{\hat E(\chi) \,:\, |\supp(\chi)|=V\}
\end{equation}
for different choices of the set $K$. Undoing the above rescaling, we then obtain the full volume and $\epsilon$-scaling by using the relation
\begin{equation}\label{eq:rescaling}
\hat E(V)=\epsilon^{-n} E_\epsilon(\epsilon^n V).
\end{equation}

\subsection{Small-volume regime}
\label{sec:small-reg}

A common point for all the following choices of $K$ is the behaviour of the scaled energy $\hat E(V)$ when $V\le 1$.

In this case, the lower bound is a consequence of the isoperimetric inequality, that is
$$
\hat E(\chi)\gtrsim \mathcal{H}^{n-1}(\p\,\supp(\chi))\gtrsim V^\frac{n-1}{n}.
$$
For the upper bound we can take e.g. $u$ to be equal to $A x$ in a ball of radius $r$, for some $A\in K$, then matching the zero boundary conditions thanks to a cut-off argument on $B_{2r}\setminus B_r$ with $r=V^\frac{1}{n}$.
Thus, by taking e.g. $\chi=A\chi_{B_{2r}}$,
$$
\hat E(u,\chi)\lesssim V+V^\frac{n-1}{n}\lesssim V^\frac{n-1}{n}
$$
in this regime of $V$.

Hence, there exist two constants $C_2>C_1>0$ depending on $n$ and $K$ such that for every $V\le 1$ we have
\begin{equation}\label{eq:small-vol}
C_1 V^\frac{n-1}{n} \le \hat E(V) \le C_2 V^\frac{n-1}{n}.
\end{equation}
From \eqref{eq:rescaling} we obtain the corresponding scaling law for $E_\epsilon(V)$;
\begin{equation}\label{eq:small-vol-eps}
C_1 \epsilon V^\frac{n-1}{n} \le E_\epsilon(V) \le C_2 \epsilon V^\frac{n-1}{n}
\end{equation}
for every $V\le\epsilon^n$.

\subsection{Elastic energy and Fourier multipliers}

In deducing the lower scaling bounds and in order to effectively explore the effects of anisotropy, in this article we will often work in frequency space.
Therefore, it is convenient to express the elastic energy in terms of the Fourier transform 
\begin{align*}
\hat{\chi}(k):=\F \chi(k):= \int\limits_{\R^n} e^{-i k\cdot x} \chi(x) dx
\end{align*}
of the function $\chi$.
The result below is the analogue of \cite[Lemma 4.1]{RT22} in the case of the full space Fourier transform.

\begin{lem}\label{lem:el-chi-fourier}
Let $E_{el}$ be as in \eqref{eq:el-chi} and $\chi\in BV(\R^n;K_0)$.
Then, there holds
\begin{equation}\label{eq:el-chi-fourier}
E_{el}(\chi)\geq \sum_{j=1}^n \sum_{\ell\neq j}\int_{\R^n}\frac{k_\ell^2}{|k|^2}|\hat\chi_{j,j}|^2dk.
\end{equation}
\end{lem}

\begin{proof}
By Plancherel's theorem, the elastic energy can be expressed as
\begin{equation}\label{eq:lem-el-en}
E_{el}(u,\chi)=\int_{\R^n}|\hat u\otimes ik-\hat\chi|^2dk.
\end{equation}
Using that 
\begin{align*}
\inf_{u\in H^1(\R^n;\R^n)}\int_{\R^n}|\nabla u-\chi|^2dx \geq \inf_{u\in \dot{H}^1(\R^n;\R^n)}\int_{\R^n}|\nabla u-\chi|^2dx,
\end{align*}
we compute the Euler-Lagrange equation associated to the minimum problem
$$
\inf_{u\in \dot{H}^1(\R^n;\R^n)}\int_{\R^n}|\nabla u-\chi|^2dx
$$
on the frequency space, that is
$$
(\hat u\otimes ik)k=\hat\chi k, \quad k\in\R^n.
$$
This is solved by $\hat u_j=-i\frac{k_j}{|k|^2}\hat\chi_{j,j}$ which, inserted in \eqref{eq:lem-el-en}, yields the desired result.
\end{proof}

We next introduce some notation. For $0<\mu<1$ and $\mu'>0$, we define
$$
C_{j,\mu,\mu'}=\Big\{k\in\R^n \,:\, \sum_{\ell\neq j}k_\ell^2\le\mu^2|k|^2,\, |k|\le\mu'\Big\}.
$$
Correspondingly, we define the associated Fourier multipliers $\chi_{j,\mu,\mu'}(D)$, where $\chi_{j,\mu,\mu'}$ are positive $C^\infty(\R^n\setminus\{0\})$ functions such that $\chi_{j,\mu,\mu'}(k)=1$ for $k\in C_{j,\mu,\mu'}$ and $\chi_{j,\mu,\mu'}(k)=0$ for $k\in \R^n\setminus C_{j,2\mu,2\mu'}$, satisfying decay conditions as in Marcinkiewicz's multiplier theorem, see for instance \cite[Corollary 6.2.5]{Grafakos}.

\subsection{Estimates in the frequency space}

The following two results are the corresponding versions of those in \cite[Section 4]{RT22} adapted to the case of the continuous Fourier transform (rather than Fourier series). They encode a first frequency localization using elastic and surface energies (Lemma \ref{lem:control-cones1}) and an iteration of this through a commutator argument (Lemma \ref{lem:commutator2}).
The proofs can be obtained by the ones from \cite{RT22} by just replacing sums with integrals. We therefore omit the repetition of these proofs.

\begin{lem}[Lemma 4.5, \cite{RT22}]\label{lem:control-cones1}
Let $n\ge2$ and $K$ be fixed. Let $\chi \in BV(\R^n, K_0)$, let $E_{el}(\chi)$ be as in \eqref{eq:el-chi} and $E_{surf}(\chi):=|D\chi|(\R^n)$.
Then there exists a constant $C>0$ depending on $n$ and $K$ such that for every $0<\mu<1$, $\mu_2>0$ there holds
$$
\sum_{j=1}^n \|\chi_{j,j}-\chi_{j,\mu,\mu_2}(D)\chi_{j,j}\|_{L^2}^2 \le C\big(\mu^{-2} E_{el}(\chi)+\mu_2^{-1}E_{surf}(\chi)\big).
$$
\end{lem}

\begin{lem}[Proposition 4.7, \cite{RT22}]\label{lem:commutator2}
Let $K$, $n$ and $\chi$ be as in the statement of Lemma \ref{lem:control-cones1}.
Let $\ell\in\{1,\dots,n\}$ and let $g$ be a polynomial such that for some $\lambda_j\in\R$, $j\in\{1,\dots,n\}$ there holds
$$
\chi_{\ell,\ell}=g\Big(\sum_{j\neq\ell}\lambda_j\chi_{j,j}\Big).
$$
Let $0<\mu<1$, $\mu',\mu''>0$ and $\tilde\mu=M\mu\mu''$ for some constant $M>1$ depending on the degree of $g$.
Then for every $\gamma\in(0,1)$ there exists a constant $C>0$ depending on $n$, $g$, $K$ and $\gamma$ such that there holds
$$
\|\chi_{\ell,\ell}-\chi_{\ell,\mu,\tilde\mu}(D)\chi_{\ell,\ell}\|_{L^2}^2 \le C \sum_{j\neq\ell}|\lambda_j|\psi_\gamma\big(\|\chi_{j,j}-\chi_{j,\mu,\mu''}(D)\chi_{j,j}\|_{L^2}^2\big)+C\|\chi_{\ell,\ell}-\chi_{\ell,\mu,\mu'}(D)\chi_{\ell,\ell}\|_{L^2}^2,
$$
where $\psi_\gamma(z)=\max\{z,z^{1-\gamma}\}$.
\end{lem}

In what follows we will apply this result for the derivation of lower bounds for higher order laminates. Here we will always apply the result with $0<\tilde{\mu}\leq\mu''\le \mu'<1$ in order to improve on the possible region of Fourier space concentration. 

We conclude this section by giving the following control of low frequencies.

\begin{lem}\label{lem:low_freq}
Let $\varphi\in L^2(\R^n)\cap L^1(\R^n)$, $0<\mu<1$, $\mu'>0$ and let $j\in\{1,\dots,n\}$. Then,
\begin{align*}
\int_{C_{j,\mu,\mu'}} |\hat{\varphi}|^2 dk \lesssim (\mu')^n \mu^{n-1} \|\varphi\|_{L^1}^2.
\end{align*}
\end{lem}

\begin{proof}
The argument follows from the $L^1-L^{\infty}$ bound for the Fourier transform. Indeed,
\begin{align*}
\int_{C_{j,\mu,\mu'}} |\hat{\varphi}|^2 dk  \leq \Big( \sup\limits_{k \in \R^{n}} |\hat{\varphi}| \Big)^2 \big|C_{j,\mu,\mu'}\big| \leq 8\|\varphi\|_{L^1}^2 (\mu')^n \mu^{n-1}.
\end{align*}
\end{proof}

\section{First order laminates: The 1+1-well case}
\label{sec:1+1}

We start our discussion of the outlined nucleation results by considering the simplest situation possible, i.e. $\#K=1$.
A more general version of this problem has been studied in \cite{KK}, in the context of geometrically linearized elasticity.

We translate the results from \cite{KK} into our context.

\begin{thm}[Theorem 2.1 \cite{KK}]
Let $K=\{A\}$ for some $A\in\diag(n,\R)$ with $\rank(A)=1$ and let $\hat E(V)$ be as in \eqref{eq:en-vol-scaled}. Then there exist two positive constants $C_2>C_1>0$ depending on $K$ and $n$ such that for every $V>0$ there holds
\begin{align*}
C_1 \hat r(V) \le \hat E(V)\le C_2 \hat r(V),
\quad \text{where} \quad
\hat r(V)=\begin{cases}
V^\frac{n-1}{n} & V\le1 , \\
V^\frac{2n-2}{2n-1} & V\ge1.
\end{cases}
\end{align*}
\end{thm}

\begin{proof}
As in \cite{KK} the proof of this result consists of two parts: An upper bound and a lower bound. The upper bound construction from \cite{KK} does not make use of gauge invariance but only the presence of a rank-one connection between the two phases. As a consequence, it also works in our context with essentially no modification.

The proof of the lower bound in \cite{KK} consists of two steps: The derivation of a localized lower bound estimate \cite[Proposition 3.1]{KK} stating that if little of the minority phase and only small perimeter is present, then a good lower bound estimate holds true.
Translated to our context, it reads:

\begin{claim}
Let $\chi \in BV(\R^n, K_0)$.
Then there exist constants $c,\alpha>0$ such that if for $R>0$
\begin{equation}\label{eq:lb-1+1-ass}
\|\chi\|_{L^1(B_R)}\le c R^{n} \quad \text{and} \quad |D\chi|(B_R)\le c R^{n-1},
\end{equation}
then
$$
\inf_{u\in H^1(\R^{n};\R^{n})}\int_{B_R}|\nabla u-\chi|^2 dx \ge c R^{- n}\|\chi\|^2_{L^1(B_{\alpha R})}.
$$
\end{claim}

Exploiting this, the full lower bound is proved by a covering argument. Since there is no difference in the covering argument, we only discuss \cite[Proposition 3.1]{KK} which in turn relies on \cite[Lemma 3.2]{KK}, its two-dimensional analogue.
For convenience of the reader and for completeness, we retrace the main ideas of the proof of \cite[Lemma 3.2]{KK}, being more specific in the parts that differ in our context.

\begin{proof}[Ideas of the proof of the claim for $n=2$:]
Without loss of generality we consider $A=e_1\otimes e_1$ and $R=1$. Further, as in \cite{KK} we fix $\alpha = \frac{1}{5}$. 
Following \cite{KK}, we now argue by contradiction, that is letting $\mu:=\|\chi\|_{L^1(B_\alpha)}$ we assume that there exists $u\in H^1(\R^2;\R^2)$ such that
\begin{equation}\label{eq:lb-1+1-cont}
\|\p_1u_1-\chi_{1,1}\|_{L^2(B_1)}+\|\p_2u_2\|_{L^2(B_1)}+\|\p_2u_1\|_{L^2(B_1)}+\|\p_1u_2\|_{L^2(B_1)}\le c_2\mu,
\end{equation}
for some constant $c_2$ arbitrarily small, and split this into several steps.

\emph{Steps 1 and 2:}
We consider three rectangles $Q^{(i)}=\big[l^{(i)}_1,l^{(i)}_2\big]\times[-h,h]$ with $l^{(i)}_2=l^{(i+1)}_1$ such that $l_2^{(3)}-l_1^{(1)},2h\le 1$ and
$$
Q^{(2)}\supset B_\alpha, \quad Q^{(i)}\subset B_{3\alpha},\  i \in \{1,2,3\}.
$$
We also denote $I_1^{(i)}=[l_1^{(i)},l_2^{(i)}]$, $ i \in \{1,2,3\}$, $I_1=[l_1^{(1)},l_2^{(3)}]$ and $I_2=[-h,h]$.
Up to small translations of $Q^{(i)}$ exploiting \eqref{eq:lb-1+1-ass} and \eqref{eq:lb-1+1-cont} one proves $\|\chi\|_{L^1(\p Q^{(i)})}$, $|D\chi|(\p Q^{(i)})=0$ and that
\begin{equation}\label{eq:no-conc}
\int_{\p Q^{(i)}}|\nabla u|^2=\int_{\p Q^{(i)}}|\nabla u-\chi|^2 \le c_2 \mu.
\end{equation}
Moreover, by possibly passing from $u_1$ to $u_1- \langle  u_1 \rangle_{I_1\times \{h\}}$ we may assume that $\langle u_1\rangle_{I_1\times \{h\} } =0$. 

\emph{Step 3 and 4:} We show that $u_1$ is close to zero in mean inside $Q^{(i)}$ for $ i \in \{1,2,3\}$.
We begin with proving that $u_1$ %and $u_2$ are
is close to $0$ %respectively 
on the horizontal %and vertical 
boundaries of $Q^{(i)}$.
Indeed for every $x\in I_1^{(i)}$, by the fundamental theorem of calculus, H\"older's inequality and \eqref{eq:no-conc}
\begin{align}
\label{eq:first}
|u_1(x,\pm h)-\langle  u_1\rangle_{I_1^{(i)}\times\{\pm h\}}|\le\|\p_1 u_1\|_{L^2(I_1^{(i)}\times\{\pm h\})}\le c_2\mu.
\end{align}
Moreover, for some $x\in I_1^{(i)}$ there holds
$$
|\langle  u_1 \rangle_{I_1^{(i)}\times\{h\}}-\langle u_1 \rangle_{I_1^{(i)}\times\{-h\}}|\le \int\limits_{I_1^{(i)}}\int_{-h}^h |\p_2 u_1(x,t)|dtdx\le c_2\mu.
$$
%[I am a bit confused about the off-diagonal derivative.]
%Working analogously for the other component 
%Up to a translation (of the type $u-Fx$) we obtain
Due to our normalization and \eqref{eq:first}, we thus obtain that
$$
\|u_1\|_{L^\infty(I_1^{(i)}\times\{\pm h\})} \le 2c_2\mu.
$$
Combined with the off-diagonal bounds in \eqref{eq:lb-1+1-cont} this yields
\begin{equation}\label{eq:mean-cont}
|\langle u_1\rangle_{Q^{(i)}}|\le 3c_2\mu, \  i \in \{1,2,3\}.
\end{equation}

\emph{Step 5:} Finally, let $x_1^*\in I_1^{(1)}$ be such that $\langle u_1\rangle_{Q^{(1)}}=\frac{1}{I_2}\int_{I_2}u_1(x_1^*,t)dt$. Then, using that $u(x_1,x_2)= u(x_1^{\ast},x_2) + \int\limits_{x_1^{\ast}}^{x_1} \p_1 u_1(t,x_2)dt$, we have
\begin{align*}
\langle u_1\rangle_{Q^{(3)}} \gtrsim \langle u_1\rangle_{Q^{(1)}} + \mu + \langle\int_{x_1^*}^{\cdot} \p_1u_1(t,\cdot)-\chi(t,\cdot) dt\rangle_{Q^{(3)}},
\end{align*}
which yields from \eqref{eq:lb-1+1-cont} by taking $c_2$ small enough, that
$$
\langle u_1\rangle_{Q^{(3)}}-\langle u_1\rangle_{Q^{(1)}} \gtrsim \mu.
$$
This contradicts \eqref{eq:mean-cont} by further reducing $c_2$ if needed and the claim is proved.
\end{proof}

The remainder of the argument for Theorem \ref{thm:KK} then follows as in \cite{KK}.
\end{proof}

\section{First order laminates: The 2+1-wells case}
\label{sec:2+1}

In this section, we discuss the proof of Theorem \ref{thm:2wells}, i.e. the setting in which
\begin{align}
\label{eq:two_wells_proof}
K=\{A,B\},
\quad \text{with} \quad
A= \begin{pmatrix} -\lambda & 0 \\ 0 & 0\end{pmatrix},\,
B= \begin{pmatrix} 1-\lambda & 0 \\ 0 & 0 \end{pmatrix}
\end{align}
for some $\lambda\in(0,1)$ fixed, with ``austenite" given by the zero matrix.

In this setting, after the normalization outlined above, we seek to prove the following bounds:

\begin{thm}
\label{thm:2well_norm}
Let $\hat E(V)$ be as in \eqref{eq:en-vol-scaled} and let $K$ be as in \eqref{eq:two_wells_proof}.
Then there exist two positive constants $C_2>C_1>0$ depending on $K$ such that for every $V>0$ there holds
\begin{align*}
C_1 \hat r(V) \le \hat E(V) \le C_2 \hat r(V),
\quad \text{where} \quad
\hat r(V)= \begin{cases}
V^\frac{1}{2} & \text{if } V\le1, \\
V^\frac{3}{5} & \text{if } V>1.
\end{cases}
\end{align*}
\end{thm}

Rescaling this as in \eqref{eq:rescaling} then implies the claim of Theorem \ref{thm:2wells}. 

In order to infer these bounds, we argue in two steps: First, in Proposition \ref{prop:2wells-lb} we deduce lower bounds. Next, in Proposition \ref{prop:2wells-ub}, we provide an upper bound construction. Combining these observations with the small volume estimates from Section \ref{sec:small-reg} and the normalization arguments from Section \ref{rmk:normal} then implies the claim from Theorems \ref{thm:2wells} and \ref{thm:2well_norm}.

\subsection{Lower bounds}
In deducing the lower bounds, we first seek to prove the following proposition:

\begin{prop}\label{prop:2wells-lb}
Let $\hat E(V)$ be as in \eqref{eq:en-vol-scaled} and let $K$ be as above. Then for every $V>1$ there holds
\begin{align}
\label{eq:two_wells}
\hat{E}(V) \gtrsim V^{3/5}.
\end{align}
\end{prop}

\begin{proof}
Let $\mu,\mu_2>0$ be two arbitrary constants which are to be determined below.
We estimate as follows by means of Lemmas \ref{lem:control-cones1} and \ref{lem:low_freq}: 
\begin{align*}
\min\{\lambda^2, (1-\lambda)^2\} V &\leq \|\chi_{1,1}\|_{L^2}^2 \leq \|\chi_{1,\mu,\mu_2}(D)\chi_{1,1}\|_{L^2}^2 +\|\chi_{1,1}- \chi_{1,\mu,\mu_2}(D) \chi_{1,1}\|_{L^2}^2\\
&\leq \max\{\lambda^2, (1-\lambda)^2\} 8 \mu_2^2 \mu V^2 + C\mu^{-2} E_{el}(\chi) + C\mu_2^{-1} E_{surf}(\chi)\\
& \le \max\{\lambda^2, (1-\lambda)^2\} 8 \mu_2^2 \mu V^2 +  C(\mu^{-2} + \mu_2^{-1}) \hat{E}(\chi).
\end{align*}
With the choice $\mu=\min\left\{\frac{\lambda^2}{(1-\lambda)^2}, \frac{(1-\lambda)^2}{\lambda^2} \right\} \frac{1}{16\mu_2^2 V}$ we obtain
\begin{align*}
\min\{\lambda^2, (1-\lambda)^2\} V \le C'(\mu_2^4 V^2+\mu_2^{-1})\hat E(\chi),
\end{align*}
for some $C'>0$.
We further optimize the right hand side in $\mu_2$ which yields $\mu_2 \sim V^{-\frac{2}{5}}$. Inserting this into the estimate and rearranging then yields the claim.
\end{proof}

\subsection{Upper bounds}
We now deal with an upper bound construction. For this we present two types of constructions: One directly using branchings in thin long domains aligned with the direction of lamination, the other exploiting the flexibility of the domain more effectively by working in diamond- or lens-shaped domains.
Since the latter are also observed in experiments \cite{TX90, Nie17,Schwa21}, we present the diamond-shaped constructions in the main body of the text and postpone the rectangular ones to the appendix. Diamond-shaped constructions which in addition incorporate  additional determinant constraints have been used in \cite{C}.

\begin{prop}
\label{prop:2wells-ub}
Let $\hat{E}(u,\chi)$ be as in \eqref{eq:en-scaled} and let $K$ be as in \eqref{eq:two_wells_proof}.
Let $V>1$ be given and let $\Omega=\conv\Big(\Big\{(0,0),\Big(\lambda L,\frac{H}{2}\Big),\Big(\lambda L,-\frac{H}{2}\Big),(L,0)\Big\}\Big)$ with $H>L>1$ and $HL=V$.
Then there exist $u\in W_0^{1,\infty}(\Omega;\R^2)$ and $\chi\in BV(\R^2;K_0)$ with $\supp(\chi)=\Omega$ such that
$$
\hat E(u,\chi)\lesssim V^\frac{3}{5}.
$$
\end{prop}

\begin{proof}%[Proof of Proposition \ref{prop:2wells-ub}]
%Consider the rhombus
%$$
%\Omega=\conv\Big(\Big\{(0,0),\Big(\lambda L,\frac{H}{2}\Big),\Big(\lambda L,-\frac{H}{2}\Big),(L,0)\Big\}\Big).
%$$
%$$
%\Omega=\Big\{(x_1,x_2)\in\R^2\,:\,\Big|x_1-\frac{L}{2}\Big|+\frac{L}{H}|x_2|\le\frac{L}{2}\Big\}.
%$$
We consider $u:\R^2\to\R^2$ with $u_2\equiv0$ and $u_1$ defined by integration in order to attain zero boundary condition on $\p\Omega$ and having $\p_1u_1=1-\lambda$ in $\Omega\cap\big((0,\lambda L)\times\R\big)$ and $\p_1u_1=-\lambda$ in $\Omega\cap\big((\lambda L,L)\times\R\big)$, that is
$$
u_1(x_1,x_2)=\begin{cases}\displaystyle
(1-\lambda)\Big(x_1-\frac{2\lambda L}{H}|x_2|\Big) & \frac{2\lambda L}{H}|x_2|\le x_1\le \lambda L,\\ \displaystyle
\lambda\Big(-x_1-\frac{2(1-\lambda)L}{H}|x_2|+L\Big) & \lambda L\le x_1\le L-\frac{2(1-\lambda)L}{H}|x_2|,\\
0 & \text{otherwise}.
\end{cases}
$$
In particular, $u\in W_0^{1,\infty}(\Omega;\R^2)$.
Let $\chi\in BV(\R^2;K_0)$ be the projection of $\nabla u$ onto $K_0$. Noticing that $\p_1u_1\equiv\chi_{1,1}$, we obtain
$$
E_{el}(u,\chi)=\int_\Omega|\p_2 u_1|^2dx\le\frac{L^3}{2H}
\quad \text{and} \quad
E_{surf}(\chi)\le H+\sqrt{\frac{H^2}{4}+L^2}.
$$
Hence,
$$
\hat E(u,\chi)\lesssim\frac{L^3}{H}+H,
$$
and by an optimization argument we obtain $L\sim H^\frac{2}{3}$ which implies $V\sim H^\frac{5}{3}$ and the result follows.
\end{proof}

\begin{figure}[t]
\includegraphics{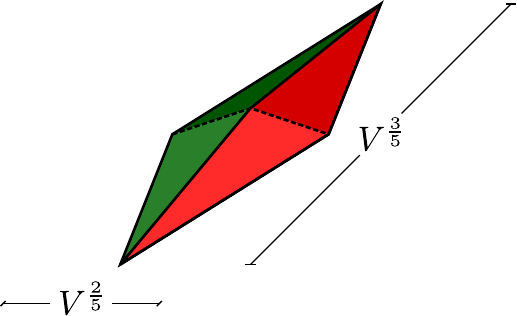}

\caption{Example of (the non-trivial component of) a two-dimensional inclusion which attains the upper bound.
Different colors represent different phases, black lines are the jump set of the phase-indicator function; on dashed lines the gradient jumps inside the same phase.}

\label{fig:lens2+1}
\end{figure}

\begin{rmk}
\label{rmk:info}
Note that the lower bound of Proposition \ref{prop:2wells-lb} is ansatz-free and therefore does not give any information on optimal structures $u$ and $\chi$.
On the other side, the upper bound of Proposition \ref{prop:2wells-ub} provides some information in that direction.
In particular, we know that we can achieve the optimal energy-scaling with structures supported in ``thin'' domains of height $V^\frac{3}{5}$ and width $V^\frac{2}{5}$. We however expect that the analysis of the optimal shape requires a more detailed study of the Euler-Lagrange equations associated with our energies.
\end{rmk}

\subsection{Proof of Corollary \ref{cor:lamination_1}}

Finally, we provide the proof of Corollary \ref{cor:lamination_1} which allows us to generalize the previous scaling law to arbitrary dimensions.

\begin{proof}[Proof of Corollary \ref{cor:lamination_1}]

We split the proof into two steps:

\emph{Step 1: Lower bounds.}
The results in the small volume regimes directly follow from the isoperimetric considerations from Section \ref{sec:small-reg}.
In the case $V>1$ (in the normalized setting), retracing the proof of Proposition \ref{prop:2wells-lb} in dimension $n$, we derive
$$
V\lesssim\mu^{n-1}\mu_2^n V^2+(\mu^{-2}+\mu_2^{-1})\hat E(\chi)
$$
yielding, after optimization, the lower bound
$$
V^\frac{3n-3}{3n-1}\lesssim \hat E(\chi).
$$

\emph{Step 2: Upper bounds.}
Again we first provide a lens-shaped construction for the upper bound analogously as done in the two-dimensional case in the proof of Proposition \ref{prop:2wells-ub}. A second argument in a thin rectangle is included in the appendix.
Our $n$-dimensional construction consists of a diamond-shaped inclusion whose lamination direction is ``very thin'' with respect to the others.
We consider the case $K=\{\pm e_1\otimes e_1\}\subset\R^{n\times n}$ from which one can recover the general case with a change of variables.

For $H>L>1$ consider the domain
$$
\Omega=\conv\Big(\Big\{\pm\frac{L}{2}e_1,\Big\{\pm\frac{H}{2}e_j\Big\}_{j=2}^n\Big\}\Big).
$$
We consider $u:\R^n\to\R^n$ with $u_j\equiv0$ for $j\in\{2,\dots,n\}$ and
$$
u_1(x)=\begin{cases}
|x_1|+\frac{L}{H}(|x_2|+\dots+|x_n|)-\frac{L}{2} & x\in\Omega, \\
0 & \text{otherwise.}
\end{cases}
$$
Notice that $u\in W^{1,\infty}_0(\Omega;\R^n)$ with $\|\nabla u\|_{L^\infty}\lesssim1$.
Let $\chi$ denote the projection of $\nabla u$ onto $K_0$. Then we have $\chi=e_1\otimes e_1$ on $\Omega\cap\big((0,+\infty)\times\R^{n-1}\big)$ and $\chi=-e_1\otimes e_1$ on $\Omega\cap\big((-\infty,0)\times\R^{n-1}\big)$.
Hence,
$$
E_{el}(u,\chi)\lesssim|\Omega|\frac{L^2}{H^2}\lesssim L^3 H^{n-3}
\quad \text{and} \quad
E_{surf}(\chi)\lesssim H^{n-1}.
$$
This implies that $\hat E(u,\chi)\lesssim L^3 H^{n-3}+H^{n-1}$. Optimizing this we get $L\sim H^\frac{2}{3}$ which yields the desired result from the volume constraint.
\end{proof}

\section{A second order laminate: An example of a 4+1 setting}
\label{sec:second_lam}

We next proceed to the analysis of second order laminates. To this end we consider the following differential inclusion:
\begin{align}
\label{eq:ex_incl2}
\nabla u \in \begin{pmatrix} 
-\chi_1 -\chi_2 + \chi_3 + \chi_4 & 0 \\ 0 & -2\chi_1 + \chi_2 + 2 \chi_3 - \chi_4
\end{pmatrix},
\end{align}
which is the indicator function formulation of the differential inclusion from Section \ref{sec:intro_second_order}. In this section, it is our main objective to prove the claim of Theorem \ref{thm:second_intro}.

We recall that in this setting, in particular, the zero matrix belongs to $K^{lc}$ and can be obtained as
$$
\frac{1}{2}\Big(\frac{1}{3}A_1+\frac{2}{3}A_2\Big)+\frac{1}{2}\Big(\frac{1}{3}A_3+\frac{2}{3}A_4\Big)=\textbf{0},
$$
that is the zero matrix is a second order laminate of $K$.
Building on the differential inclusion from \eqref{eq:ex_incl2}, a useful way to write the elastic energy is the following
$$
E_{el}(u,\chi)=\int_{\R^2}\Big|\nabla u-\chi_1 A_1-\chi_2 A_2-\chi_3 A_3-\chi_4 A_4\Big|^2 dx
$$
with $\chi_j\in BV(\R^2;\{0,1\})$ and $\chi_1+\chi_2+\chi_3+\chi_4\le1$.
In terms of the phase indicator $\chi$, we have $\chi=\sum_{j=1}^4\chi_j A_j$, hence $\chi_{1,1}=-\chi_1-\chi_2+\chi_3+\chi_4$ and $\chi_{2,2}=-2\chi_1+\chi_2+2\chi_3-\chi_4$.

%For later use we denote $\Omega=\supp(\chi)$ and we notice that $\chi_\Omega=\chi_1+\chi_2+\chi_3$.
With this notation we can write one component in terms of the other as follows
\begin{equation}\label{eq:nonlinear}
\chi_{1,1}=g(\chi_{2,2}),
\end{equation}
where $g$ is a polynomial of degree $3$.
The choice of $g$ is not unique and it can be made explicit e.g. by interpolation.
An example of such a polynomial is $g(t)=\frac{1}{2}t^3-\frac{3}{2}t$.
Notice that the condition $g(0)=0$ ensures that \eqref{eq:nonlinear} is globally satisfied, in particular also outside the inclusion domain.
As before, by rescaling and the small volume energy bounds from Section \ref{sec:prelim}, it suffices to prove the following result:

\begin{thm}
\label{thm:second}
Let $\hat E(V)$ be as in \eqref{eq:en-vol-scaled} and let $K$ be as in \eqref{eq:4_wells}.
Then, there exist two positive constants $C_2>C_1>0$ depending on $K$ such that for every $V>0$ there holds
\begin{align*}
C_1 \hat r(V) \le \hat E(V) \le C_2 \hat r(V),
\quad \text{where} \quad
\hat r(V)= \begin{cases}
V^\frac{1}{2} & \text{if } V\le1, \\
V^\frac{5}{7} & \text{if } V>1.
\end{cases}
\end{align*}
\end{thm}

As above, we split this into two steps: In Proposition \ref{prop:3wells-lb} we first deduce corresponding lower bounds and in Proposition \ref{prop:3wells-ub} we provide matching upper bounds. Using the considerations from Sections \ref{rmk:normal} and \ref{sec:small-reg} then implies the desired claim of Theorem \ref{thm:second}.

\subsection{Lower bounds}

We first give a lower bound for the large inclusion regime.

\begin{prop}\label{prop:3wells-lb}
Let $\hat E(V)$ be as in \eqref{eq:en-vol-scaled} and let $K$ be as in \eqref{eq:4_wells}.
Then for every $V>1$ there holds
\begin{align}
\label{eq:3_wells}
\hat{E}(V) \gtrsim V^\frac{5}{7}.
\end{align}
\end{prop}

\begin{proof}
Let $\mu,\mu_2 \in (0,1)$ be arbitrary constants to be determined and let $\mu_3=M\mu\mu_2$ for some $M>1$ depending on $g$.
We begin by observing that
$$
V= \|\chi_{1,1}\|_{L^2}^2\le 2\|\chi_{1,\mu,\mu_3}(D)\chi_{1,1}\|_{L^2}^2+2\|\chi_{1,1}-\chi_{1,\mu,\mu_3}(D)\chi_{1,1}\|_{L^2}^2.
$$
Applying Lemma \ref{lem:commutator2} with $\mu'=\mu''=\mu_2$ thanks to relation \eqref{eq:nonlinear} and by Lemmas \ref{lem:control-cones1} and \ref{lem:low_freq} we then arrive at
\begin{align*}
V &\le 8\mu\mu_3^2 V^2+C\|\chi_{1,1}-\chi_{1,\mu,\mu_2}(D)\chi_{1,1}\|_{L^2}^2+C_\gamma\psi_\gamma\big(\|\chi_{2,2}-\chi_{2,\mu,\mu_2}(D)\chi_{2,2}\|_{L^2}^2\big) \\
& \le 8\mu\mu_3^2 V^2 + C_\gamma(\mu^{-2}+\mu_2^{-1})\psi_\gamma\big(\hat E(\chi)\big).
\end{align*}
Now, recalling that $\mu_3\sim\mu\mu_2$, we choose $\mu_2$ such that $8\mu^3\mu_2^2V^2=\frac{V}{2}$, that is $\mu_2\sim V^{-\frac{1}{2}}\mu^{-\frac{3}{2}}$.
Such a choice yields
$$
V\le C_\gamma(\mu^{-2}+V^\frac{1}{2}\mu^\frac{3}{2})\psi_\gamma\big(\hat E(\chi)\big),
$$
where we have possibly increased the constant $C_\gamma$.
Optimizing we choose $\mu\sim V^{-\frac{1}{7}}$ and therefore obtain that
$$
V\le C_\gamma V^\frac{2}{7}\psi_{\gamma}\big(\hat E(\chi)\big)
$$
which implies
\begin{align}
\label{eq:final_nearly}
V^\frac{5}{7}\lesssim\psi_{\gamma}\big(\hat E(\chi)\big).
\end{align}
From the condition $V>1$ and the isoperimetric estimate $\hat E(\chi)\gtrsim V^\frac{1}{2}$. 
As a consequence, $\psi_{\gamma}\big(\hat E(\chi)\big) \lesssim \hat E(\chi)$ and \eqref{eq:final_nearly} implies the desired lower bound.
\end{proof}

\subsection{Upper bounds}

We next turn to an upper bound construction in the large volume regime for which we show a matching upper bound.

\begin{prop}\label{prop:3wells-ub}
Let $\hat E(u,\chi)$ be as in \eqref{eq:en-scaled}, let $K$ be as in \eqref{eq:4_wells} and
let $V>1$ be given.
Then there exist $\Omega\subset\R^2$ compact with Lipschitz boundary and $|\Omega|=V$, $u\in W_0^{1,\infty}(\Omega;\R^2)$ and $\chi\in BV(\R^2;K_0)$ with $\supp(\chi)=\Omega$ such that
$$
\hat E(u,\chi)\lesssim V^\frac{5}{7}.
$$
\end{prop}

As above, we again present two possible proofs of this result: A construction in a thin rectangle with double branching (given in Section \ref{app:second} in the Appendix) and a construction in a thin diamond/lens with simple branching. The latter exploits the more flexible domain geometry by consisting of a lens/diamond shape.
We highlight that the self-similar argument exploited for building the geometry of the construction below takes its inspiration from those used in \cite{KW16,KKO13}.

\begin{proof}[Proof of Proposition \ref{prop:3wells-ub} by means of a construction with a lens-type shape]

Consider the rhombus
$$
\Omega=\conv\Big(\Big\{\Big(-\frac{L}{2},0\Big),\Big(0,\frac{H}{2}\Big),\Big(0,-\frac{H}{2}\Big),\Big(0,\frac{L}{2}\Big)\Big\}\Big).
$$

As a first step we consider the construction given in Proposition \ref{prop:2wells-ub} with gradients
$$
B_1=\frac{1}{3}A_1+\frac{2}{3}A_2
\quad \text{and} \quad
B_2=\frac{1}{3}A_3+\frac{2}{3}A_4.
$$
This corresponds to the macroscopic deformation. This macroscopic deformation will be replaced by a microscopic deformation which is achieved by a construction of fine scale oscillations between phases $A_1$ and $A_2$ on the left part of the domain %, i.e. $\Omega\cap\big((-\infty,0)\times\R)\big)$,
and between $A_3$ and $A_4$ on the right part. % $\Omega\cap\big((0,+\infty)\times\R)\big)$.
We work in several steps.

%For the sake of clarity of exposition we divide $\Omega$ into triangles using the following notation
%\begin{align*}
%T_{I} & :=\Omega\cap\big((0,+\infty)\times(0,+\infty)\big),
%& T_{II}:=\Omega\cap\big((-\infty,0)\times(0,+\infty)\big), \\
%T_{III} &:=\Omega\cap\big((-\infty,0)\times(-\infty,0)\big),
%& T_{IV}:=\Omega\cap\big((-\infty,0)\times(0,+\infty)\big).
%\end{align*}

\emph{Step 1: Definition of a macroscopic state.}
We consider $w$ to be obtained analogously as in the proof of Proposition \ref{prop:2wells-ub} with $K$ replaced by the set $\{B_1,B_2\}$, that is $w=(w_1,0)$ with
$$
w_1(x_1,x_2)=\begin{cases}
-\frac{L}{2}+|x_1|+\frac{L}{H}|x_2| & x\in\Omega, \\
0 & \text{otherwise.}
\end{cases}
$$
The function $w$ will play the role of a macroscopic state. We highlight that the construction from Proposition \ref{prop:2wells-ub} does not give deformations exactly equal to $B_1, B_2$ but involves slight perturbations of these (see the discussion in Step 4 below).

\emph{Step 2: Branching building block.}
We now define a single tree of a branching construction on a general rectangle $R:=[0,\ell]\times[0,h]$ with $\ell>4 h$.
By applying \cite[Lemma 3.2]{RT22} with inverted roles of $x_1$ and $x_2$ and with $N=1$, we obtain $v^{j,R}\in W^{1,\infty}(R;\R^2)$ such that $v^{j,R}=B_j x$ on $\p R$ and
\begin{equation}\label{eq:bra-block}
\hat E(v^{j,R},\chi^{j,R})\lesssim \frac{h^3}{\ell}+\ell,
\end{equation}
where $\chi^{j,R}$ is the projection of $\nabla v^{j,R}$ onto the set $\{A_1,A_2\}$ or $\{A_3,A_4\}$ if $j=1$ or $2$ respectively.
Note that, by translation invariance, the same construction (still matching $B_j x$ at the boundary) can be obtained in rectangles $x+R$ for any $x\in\R^2$.

\begin{figure}[t]
\includegraphics{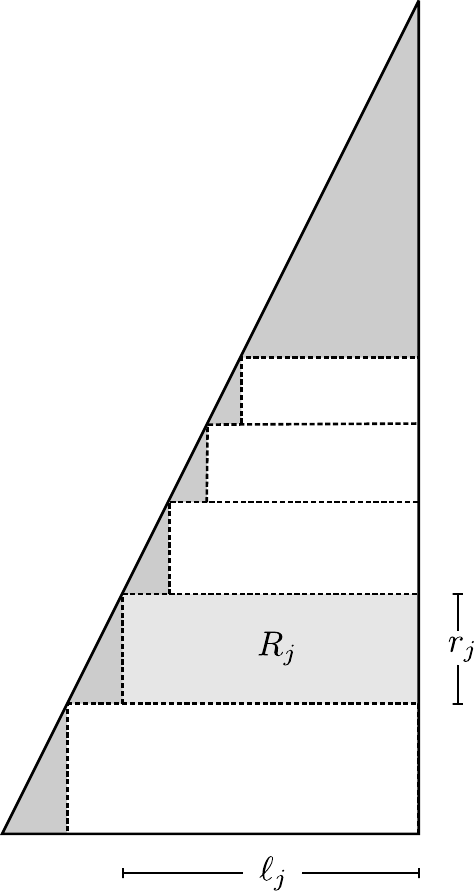}
\caption{Subdivision of $T$ into rectangles $R_j$.
The shaded region is $T\setminus T_0$.}
\label{fig:lens4+1}
\end{figure}

\emph{Step 3: Subdivision of the domains into rectangles.}
We subdivide the domain $\Omega$ into rectangles having the same ratio between width and height.
The correct ratio will be then obtained in Step 5 via optimization.
We work in the triangular subdomain $T=\conv\big(\big\{(-\frac{L}{2},0),(0,\frac{H}{2}),(0,0)\big\}\big)$ recovering the full subdivision symmetrically.

We fix $r>0$ a small parameter to be determined later.
We cut $T$ into horizontal slices at height $h_j$, with $h_0=0$, $h_1=r$ and $h_{j+1}>h_j$.
In each of these slices we consider the maximal rectangle contained in $T$, that is
\begin{equation}\label{eq:rectangles}
R_j:=[-\ell_j,0]\times[h_j,h_j+r_j],
\quad \text{with} \quad
r_j:=h_{j+1}-h_j,
\quad \ell_j:=\sup\{\ell\,:\, (-\ell, h_{j+1})\in T\}.
\end{equation}
We choose $r_j$ so that the rectangles $R_j$ have the same ratio between width and height.
Solving recursively $r_j=\frac{r}{\ell_0}\ell_j$, we find
$$
r_j=r\Big(1-\frac{2}{H}r\Big)^j
\quad \text{and} \quad
\ell_{j}=\frac{L}{2}\Big(1-\frac{2}{H}r\Big)^{j+1}.
$$
We make this subdivision for $0\le j\le j_0$, with $j_0$ being the largest index such that $h_{j_0}<\frac{H-L}{2}$.
In the end, we also define $T_0=\bigcup_{j=0}^{j_0} R_j$ (see Figure \ref{fig:lens4+1}).

\emph{Step 4: Definition of $u$.}
We define $u$ to be a fine-scale branching oscillation inside $R_j$ and being equal to the macroscopic state $w$ on $T\setminus T_0$.
To do so, we slightly modify the functions obtained in Step 2 with an affine perturbation in order to match the macroscopic state $w$ at $\p R_j$ (which is close to an affine function with gradient $B_1$ or $B_2$ but is not exactly equal to one these, see the comment in Step 1 above).
Thus, for every $x\in T$ we define
$$
u(x)=\begin{cases}
v^{1,R_j}(x)+(\frac{L}{H}x_2-\frac{L}{2},0) & x\in R_j,\, 0\le j\le j_0, \\
w(x) & \text{otherwise.}
\end{cases}
$$
Reasoning symmetrically we obtain a construction on the whole domain $\Omega$.
As usual we will denote with $\chi$ the projection of $\nabla u$ onto $K_0$.

\emph{Step 5: Energetic cost and optimization of parameters.}
By symmetry we can restrict to computing the total energy on $T$.
Splitting the contributions on $R_j$ and $T\setminus T_0$ we obtain
\begin{equation}\label{eq:lam+bra-en}
\begin{split}
\hat E(u,\chi) & \lesssim \int_{T\setminus T_0} \dist^2(\nabla w,K_0)dx + |D\chi|(T\setminus T_0) \\
& \quad +\sum_{j=0}^{j_0} \Big(\int_{R_j} \dist^2\Big(\nabla v^{1,R_j}+\frac{L}{H}e_1\otimes e_2,K_0\Big)dx + |D\chi^{1,R_j}|(R_j)\Big).
\end{split}
\end{equation}
We study the two contributions on the right-hand-side of \eqref{eq:lam+bra-en} separately, starting from the first one.

The term $|D\chi|(T\setminus\overline{T_0})$ is proportional to the perimeter of $\Omega$, hence $|D\chi|(T\setminus\overline{T_0})\lesssim H$.
We now need to control the measure of $T\setminus T_0$ (see Figure \ref{fig:lens4+1}) which can be estimated as follows;
$$
|T\setminus T_0|\lesssim r^2\frac{L}{H}\sum_{j\ge0}\Big(1-\frac{2}{H}r\Big)^{2j}+\frac{L^3}{H} \lesssim rL+\frac{L^3}{H}.
$$
%Inside each slide $\R\times(h_j,h_{j+1})$, it consists in a triangle of area smaller than $r_j(\ell_j-\ell_{j-1})=$.
From this and the fact that $\dist(\nabla w,K)\lesssim 1$, we obtain
\begin{equation}\label{eq:RHS1}
\int_{T\setminus T_0} \dist^2(\nabla w,K_0)dx + |D\chi|(T\setminus \overline{T_0}) \lesssim rL+\frac{L^3}{H}+H.
\end{equation}
We analyze the last term in \eqref{eq:lam+bra-en}.
Since the matrix $\frac{L}{H}e_1\otimes e_2$ does not affect the projection of $\nabla v^{1,R_j}$ onto $K_0$, from \eqref{eq:bra-block} for every $j$ we have
$$
\int_{R_j} \dist^2\Big(\nabla v^{1,R_j}+\frac{L}{H}e_1\otimes e_2,K\Big)dx + |D\chi^{1,R_j}|(R_j) \lesssim \frac{r_j^3}{\ell_j}+|R_j|\frac{L^2}{H^2}+\ell_j.
$$
By summing over $j$ we obtain
\begin{multline}\label{eq:RHS2}
\sum_{j=0}^{j_0} \Big(\int_{R_j} \dist^2\Big(\nabla v^{1,R_j}+\frac{L}{H}e_1\otimes e_2,K\Big)dx + |D\chi^{1,R_j}|(R_j)\Big) \\
\lesssim \sum_{j\ge0}\Big(\frac{r^3}{L}\Big(1-\frac{2}{H}r\Big)^{2j}+L\Big(1-\frac{2}{H}r\Big)^j\Big)+|T|\frac{L^2}{H^2}
\lesssim r^2 \frac{H}{L}+\frac{HL}{r}+\frac{L^3}{H}.
\end{multline}
Inserting \eqref{eq:RHS1} and \eqref{eq:RHS2} into \eqref{eq:lam+bra-en} we infer that
\begin{equation}\label{eq:tot-en-lam+bra}
\hat E(u,\chi) \lesssim rL+H+\frac{L^3}{H}+r^2\frac{H}{L}+\frac{HL}{r}.
\end{equation}
Optimizing the expression above in $r$ we get $r\sim L^\frac{2}{3}$ and $H\sim L^\frac{4}{3}$.
Notice that the choice $r\sim L^\frac{2}{3}$ is compatible with the condition $\ell>4h$ on rectangles $R$ of Step 2 (which for $R=R_j$ corresponds to $\ell_j>4r_j$) and therefore the construction above is well-defined.
The relation $HL\sim V$ yields $L\sim V^\frac{3}{7}$, $H\sim V^\frac{4}{7}$ and the result follows.
\end{proof}

\subsection{A three-dimensional analogue}
\label{sec:3D}

An interesting three-dimensional modification of the previous setting is the following: We consider $K=\{A_1,A_2,A_3,A_4\}\subset\R^{3\times3}$ with
\begin{align}
\label{eq:3D_analogue}
A_1= \begin{pmatrix} -1 & 0 & 0 \\ 0 & -2 & 0 \\ 0 & 0 & 0\end{pmatrix},\,
A_2= \begin{pmatrix} -1 & 0 & 0 \\ 0 & 1 & 0 \\ 0 & 0 & 0\end{pmatrix},\,
A_3= \begin{pmatrix} 1 & 0 & 0 \\ 0 & 0 & 0 \\ 0 & 0 & 2 \end{pmatrix},\,
A_4= \begin{pmatrix} 1 & 0 & 0 \\ 0 & 0 & 0 \\ 0 & 0 & -1 \end{pmatrix}.
\end{align}

\begin{figure}[t]
\includegraphics{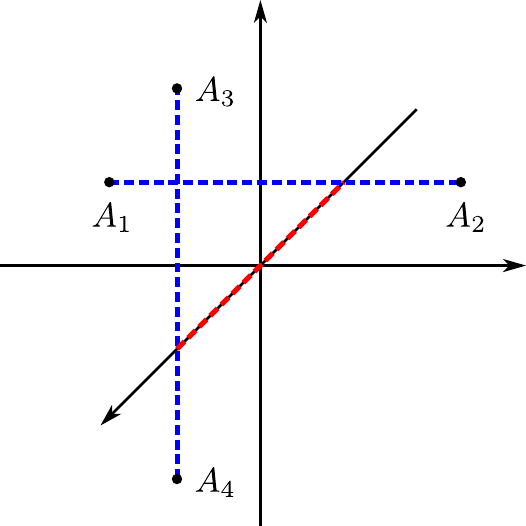}
\caption{The matrix set $K$ defined above: In blue the set $K^{(1)}\setminus K$, in red $K^{(2)}\setminus K^{(1)}$.}
\label{fig:4wells-3Da}
\end{figure}

First-order laminates are the segments $K^{(1)}=\conv(A_1,A_2)\cup\conv(A_3,A_4)$, whereas the second-order laminates consists of
$$
K^{(2)}\setminus K^{(1)}=\left\{\begin{pmatrix} \mu & 0 & 0 \\ 0 & 0 & 0 \\ 0 & 0 & 0\end{pmatrix} \,:\, |\mu|<1\right\}
$$
which in particular contains $\textbf{0}$, as can be seen in Figure \ref{fig:4wells-3Da}.

In this case we have $\chi_{1,1}=-\chi_1-\chi_2+\chi_3+\chi_4$, $\chi_{2,2}=-2\chi_1+\chi_2$ and $\chi_{3,3}=2\chi_3-\chi_4$ and therefore we have the nonlinear relation
\begin{equation}\label{eq:nonlinear-3D}
\chi_{1,1}=g(\chi_{2,2}+\chi_{3,3})
\end{equation}
with $g$ a polynomial. 
Because of the structure of the wells, we can choose $g$ to be the same polynomial as in \eqref{eq:nonlinear}.

For this setting, working analogously as done for the two-dimensional case, we have the following lower scaling bounds.

\begin{prop}
\label{prop:3D_analogue}
Let $\hat E(V)$ be as in \eqref{eq:en-vol-scaled} and let $K$ be as in \eqref{eq:3D_analogue}.
Then there exists a positive constant $C>0$ depending on $K$ such that for every $V>0$ there holds
\begin{equation}\label{eq:thm-4wells3D-1}
C \hat r(V) \le \hat E(V), %\le C_2 \hat r(V),
\quad \text{where} \quad
\hat r(V)= \begin{cases}
V^\frac{2}{3} & \text{if } V\le1, \\
V^\frac{9}{11} & \text{if } V>1.
\end{cases}
\end{equation}
Moreover, for every $\epsilon>0$ there holds
\begin{equation}\label{eq:thm-4wells3D-2}
C r_\epsilon(V) \le E_\epsilon(V), %\le C_2 r_\epsilon(V),
\quad \text{where} \quad
r_\epsilon(V)=\begin{cases}
\epsilon V^\frac{2}{3} & \text{if } V\le\epsilon^3, \\
\epsilon^\frac{6}{11}V^\frac{9}{11} & \text{if } V>\epsilon^3.
\end{cases}
\end{equation}
\end{prop}

We remark that the lower scaling bound in \eqref{eq:thm-4wells3D-1} and \eqref{eq:thm-4wells3D-2} is the same as the one obtained in \cite{KKO13} in the case of nucleation for the geometrically linearized cubic-to-tetragonal phase transition. In the geometrically linearized cubic-to-tetragonal phase transition the zero matrix also is a second-order laminate. The geometrically linearized cubic-to-tetragonal phase transition includes a geometrically linearized version of frame-indifference and thus $Skew(3)$-invariance. While our simplified model does not include this, it is rich enough to provide the same lower bound scaling behaviour. It could thus serve as a model problem in which one can possibly study finer properties of this phase transformation.

We expect that upper-bound constructions matching the lower-scaling bounds can be obtained working in the same spirit of \cite[Section 6]{KKO13}.
Since this is not the main goal of our work we omit to study it.

\begin{proof}
The lower bound is proved analogously as done in the two-dimensional case.
Let $0<\mu<1$, $\mu_2,\mu_3>0$ be as in the proof of Proposition \ref{prop:3wells-lb}.
From Lemmas \ref{lem:control-cones1}--\ref{lem:low_freq} and relation \eqref{eq:nonlinear} we have
$$
V \le 8\mu^2\mu_3^3 V^2+C_\gamma(\mu^{-2}+\mu_2^{-1})\psi_\gamma\big(\hat E(\chi)\big).
$$
Again, from the preliminary, isoperimetric lower bound $\hat E(\chi)\gtrsim V^\frac{2}{3}$ and the fact that $V\geq 1$, by fixing $\gamma$, from the inequality above we obtain
$$
V \lesssim \mu^2\mu_3^3 V^2+(\mu^{-2}+\mu_2^{-1})\hat E(\chi).
$$
Since $\mu_3\sim\mu\mu_2$, we choose $\mu_2$ sufficiently small such that $\mu^5\mu_2^3V^2\sim V$, that is $\mu_2\sim V^{-\frac{1}{3}}\mu^{-\frac{5}{3}}$.
This implies
$$
V\lesssim (\mu^{-2}+V^\frac{1}{3}\mu^\frac{5}{3})\hat E(\chi).
$$
Optimizing, we choose $\mu\sim V^{-\frac{1}{11}}$ and therefore obtain
$$
V\lesssim V^\frac{2}{11}\hat E(\chi)
$$
which gives the claim.

As above, the $\epsilon$-scaling behaviour follows by rescaling as in Section \ref{rmk:normal}.
\end{proof}

\section{An example with third order lamination}
\label{sec:third_order}

In this section, we provide the proof of Proposition \ref{prop:lower_3}.
As in the previous sections, it suffices to deduce the non-dimensionalized bounds. These read as follows:

\begin{prop}[Lower bounds for third order laminates in three dimensions]
\label{prop:3order_lower1}
Let $\hat{E}(V)$ be as in \eqref{eq:en-vol-scaled} and let $K$ be as in \eqref{eq:three_dim_3rd}.
Then, there exists a positive constant $C>0$ depending on $K$ such that for every $V>0$ it holds
\begin{align*}
\hat{E}(V) \geq C \hat{r}(V), \mbox{ where } 
\hat{r}(V)= 
\left\{
\begin{array}{ll} 
 V^{\frac{2}{3}}, & \ V \leq 1,\\
V^{\frac{6}{7}}, & \ V > 1.
\end{array} \right.
\end{align*}
\end{prop}

The main idea of the argument consists in using that $\chi_{1,1}$ determines $\chi_{2,2}, \chi_{3,3}$
and that $\chi_{2,2}$ determines $\chi_{3,3}$. Thus, an iterated commutator estimate as in Section \ref{sec:prelim} is possible and an optimization argument yields the lower bound.

\begin{proof}
We argue in several steps. As a preliminary observation, we note that it suffices to prove the bounds for $V> 1$ since the small volume setting is a direct consequence of the isoperimetric inequality, see Section \ref{sec:small-reg}. 

\emph{Step 1: First commutator bounds.} From Lemma \ref{lem:control-cones1} we directly obtain that for $\mu_2, \mu \in (0,1)$ to be fixed, it holds that
\begin{align}\label{eq:lb-3order-1}
\sum\limits_{j=1}^{3}\|\chi_{j,j}-\chi_{j,\mu,\mu_2}(D)\chi_{j,j}\|_{L^2}^2 \leq C(\mu^{-2} E_{el}(\chi) + \mu_2^{-1} E_{surf}(\chi)).
\end{align}

\emph{Step 2: $\chi_{1,1}$ determines $\chi_{2,2}$ and $\chi_{3,3}$.}
From the structure of the wells, we observe that $\chi_{2,2} = f_{1,2}(\chi_{1,1})$ and $\chi_{3,3} = f_{1,3}(\chi_{1,1})$ for some polynomials $f_{1,2}, f_{1,3}$ (see Remark \ref{rmk:poly-ex} for an explicit example).
As a consequence, we may invoke Lemma \ref{lem:commutator2} which, combined with \eqref{eq:lb-3order-1}, yields that for $0<\mu_3=M\mu \mu_2<\mu_2$ and for $j\in\{2,3\}$, it holds that
\begin{align*}
\|\chi_{j,j}-\chi_{j,\mu,\mu_3}(D)\chi_{j,j}\|_{L^2}^2 \leq C \psi_{\gamma}\big(\|\chi_{1,1}-\chi_{1,\mu,\mu_2}(D)\chi_{1,1}\|_{L^2}^2\big) + C(\mu^{-2} E_{el}(\chi) + \mu_2^{-1} E_{surf}(\chi)).
\end{align*}

Moreover, due to the fact that $V> 1$ and the isoperimetric inequality, we may drop the function $\psi_{\gamma}$ which yields, again from \eqref{eq:lb-3order-1}, that
\begin{align}\label{eq:lb-3order-2}
\|\chi_{j,j}-\chi_{j,\mu,\mu_3}(D)\chi_{j,j}\|_{L^2}^2 \leq C(\mu^{-2} E_{el}(\chi) + \mu_2^{-1} E_{surf}(\chi)), \ j\in \{2,3\}.
\end{align}

\emph{Step 3: $\chi_{2,2}$ determines $\chi_{3,3}$.} Using that $\chi_{3,3}=f_{2,3}(\chi_{2,2})$ with $f_{2,3}$ a polynomial (see Remark \ref{rmk:poly-ex}), we again invoke Lemma \ref{lem:commutator2} with $\mu'=\mu_2$, $\mu''=\mu_3$ and $\tilde\mu=\mu_4$. Hence, for $0<\mu_4=M\mu \mu_3<\mu_3$ this yields the bounds
\begin{align*}
\|\chi_{3,3}-\chi_{3,\mu,\mu_4}(D)\chi_{3,3}\|_{L^2}^2 &\leq \|\chi_{2,2}-\chi_{2,\mu,\mu_3}(D)\chi_{2,2}\|_{L^2}^2 + C(\mu^{-2} E_{el}(\chi) + \mu_2^{-1} E_{surf}(\chi)) \\
& \le 2C(\mu^{-2} E_{el}(\chi) + \mu_2^{-1} E_{surf}(\chi)),
\end{align*}
where the last inequality is a consequence of \eqref{eq:lb-3order-2}.

\emph{Step 4: Optimization and conclusion.}
With the bounds from the previous steps, we conclude that
\begin{align}
\label{eq:optimize_3}
\begin{split}
V & = \|\chi_{33}\|_{L^2}^2 \\
&\leq 2(\|\chi_{3,\mu,\mu_4}(D)\chi_{3,3}\|_{L^2}^2 + \|\chi_{3,3}-\chi_{3,\mu,\mu_4}(D)\chi_{3,3}\|_{L^2}^2)\\
& \lesssim \mu^2 \mu_4^3 V^2 + (\mu_2^{-1}+\mu^{-2})\hat{E}(\chi)
= \mu^2 (\mu^2 \mu_2)^3V^2 + (\mu_2^{-1}+\mu^{-2})\hat{E}(\chi).
\end{split}
\end{align}
Absorbing the first right hand contribution from \eqref{eq:optimize_3} into the left hand side, we choose $\mu_2 \sim \mu^{-\frac{8}{3}} V^{-\frac{1}{3}}$. As a consequence, inserting this back into \eqref{eq:optimize_3}, we arrive at
\begin{align*}
V & = \|\chi_{3,3}\|_{L^2}^2 \lesssim
(\mu^{\frac{8}{3}} V^{\frac{1}{3}} +\mu^{-2})\hat{E}(\chi).
\end{align*}
Choosing $\mu\sim V^{-\frac{1}{14}}$, we obtain that 
\begin{align*}
V \lesssim V^{\frac{1}{7}} \hat{E}(\chi),
\end{align*}
which yields the claim.
\end{proof}

\begin{rmk}\label{rmk:poly-ex}
We provide examples of nonlinear polynomials such that
\begin{equation}\label{eq:nonlinear-8-wells}
\chi_{2,2}=f_{1,2}(\chi_{1,1}), \quad \chi_{3,3}=f_{1,3}(\chi_{1,1}), \quad \chi_{3,3}=f_{2,3}(\chi_{2,2}).
\end{equation}
This can be found for instance by interpolation.
Hence, defining the numerical constants $a=1344$, $b=1440$, $c=576$, $d=5040$, the polynomials
\begin{align*}
f_{1,2}(t) &=-\frac{1}{a}t^8+\frac{11}{b}t^6+\frac{83}{c}t^4-\frac{5801}{d}t^2,\\
\quad
f_{1,3}(t) &=5\Big(\frac{1}{a}t^8-\frac{31}{b}t^6+\frac{101}{c}t^4-\frac{1781}{d}t^2\Big),\\
f_{2,3}(t) &=\frac{5}{12}t^4-\frac{7}{12}t^2,
\end{align*}
comply with \eqref{eq:nonlinear-8-wells}.
\end{rmk}

\section{An infinite-order laminate: Setting for the Tartar case}
\label{sec:tartar}

Last but not least, we turn to the proof of Theorem \ref{thm:Tartar}. 
Here the phase indicator $\chi$ has the form
$$
\chi=\begin{pmatrix}-\chi_1+\chi_3-3\chi_2+3\chi_4&0\\0&-3\chi_1+3\chi_3+\chi_2-\chi_4\end{pmatrix}
$$
with $\chi_j\in BV(\R^2;\{0,1\})$ with $\chi_1+\chi_2+\chi_3+\chi_4\le1$.
For this well structure each component $\chi_{j,j}$ determines the other by a nonlinear polynomial relation
\begin{equation}\label{eq:nonlinearT4}
\chi_{1,1}=f(\chi_{2,2}) \quad \text{and} \quad \chi_{2,2}=g(\chi_{1,1}).
\end{equation}
This is a consequence of the rank-one-incompatibility of $K$, for a detailed treatment see \cite{RT21}.

After non-dimensionalization it suffices to prove the following estimates:

\begin{thm}
\label{thm:Tartar_non}
Let $\hat{E}(V)$ be as in \eqref{eq:en-vol-scaled} and let $K$ be as in \eqref{eq:T4_wells}.
Then there exist four positive constants $C^{(1)},C^{(2)}>0$, $C_2>C_1>0$ depending on $K$ such that for every $V>0$ there holds
\begin{equation}\label{eq:thm-T4-1}
C_1 \hat r^{(1)}(V) \le \hat E(V) \le C_2 \hat r^{(2)}(V),
\quad \text{where} \quad
\hat r^{(j)}(V)= \begin{cases}
V^\frac{1}{2} & \text{if } V\le1, \\
V\exp(-C^{(j)}\log(V)^\frac{1}{2}) & \text{if } V>1.
\end{cases}
\end{equation}
\end{thm}

Similarly as in the previous sections, we split this into an upper bound construction and lower bound estimates which we provide in the next subsections.

\subsection{Upper bounds}
%
%[this I can maybe see if can be reworked on the level of $\epsilon=1$ both to simplify and to be consistent with the rest of the paper] [I think that it would be great, if you could try that. It would make things more consistent]
%
We begin by providing an upper bound construction:

\begin{prop}\label{prop:T4:ub}
Let $\hat{E}(u,\chi)$ be as in \eqref{eq:en-scaled} and let $K$ be as in \eqref{eq:T4_wells}.
Let $V>1$ be given and let $\Omega=[0,L]\times[0,H]$ with $H\le L$,

\begin{equation}\label{eq:geom-non-deg}
L\exp\big(-C\log(L)^\frac{1}{2}\big)<cH,
\end{equation}
for some universal constant $0<c<1$ and $HL=V$.
Then there exist $u\in W_0^{1,\infty}(\Omega;\R^2)$ and $\chi\in BV(\R^2;K_0)$ with $\supp(\chi)=\Omega$ such that
\begin{align*}
\hat E(u,\chi)\lesssim V \exp\big(-C^{(2)}\log(V)^\frac{1}{2}\big)
\end{align*}
for some constant $C^{(2)}>0$.
\end{prop}

The argument for this relies on a quantitative (in the volume) analysis of the constructions from \cite{W97,C99,RT21}.

\begin{proof}
Let $r>0$ be a parameter to be determined, depending on $L$ and $H$ and complying with
\begin{equation}\label{eq:non-deg}
r<c H, \quad \text{for some universal constant } c<1.
\end{equation}
For such $r$ we define the following construction:
Consider $r_j>0$ for $j\in \N$, $j\geq 2$, such that $r_{j+1}<r_j$ and $r_2<r$.
Here, we assume that $r_j$ can be expressed in terms of $r$. %, $H$ and $L$.
Let $u_{r,k}$ be obtained via the $k$-th-order lamination construction defined in Step 1 of the proof of \cite[Theorem 3.1]{W97} (see also \cite[Section 2]{RT21}).
Using the notation of \cite{RT21} we also take $\chi_{r,k}$ to be the projection of $\nabla u_{r,k}$ onto $K_0$.
In \cite{W97} (equation at page 17 with $\epsilon=1$) it is proved that
\begin{equation}\label{eq:ub-vol}
\hat E(u_{r,k},\chi_{r,k})\lesssim LH\Big(\frac{1}{2^k}+\frac{r}{L}+\sum_{j=2}^k2^{-j}\frac{r_j}{r_{j-1}}+2^{-j}\frac{1}{r_k}\Big).
\end{equation}
The same estimate can be found by reworking the lines of \cite{RT21} as well.
A good choice for the length scales $r_j$ comes from an optimization procedure; by imposing $\frac{r}{L}\sim\frac{r_j}{r_{j-1}}$ we get
$r_j\sim\frac{r^j}{L^{j-1}}$.
With this choice \eqref{eq:ub-vol} reduces to
$$
\hat E(u_{r,k},\chi_{r,k})\lesssim H\big(L 2^{-k}+r+L^k r^{-k}\big).
$$
Optimizing $r$ in terms of $k$ and $L$ leads to $r\sim L^\frac{k}{k+1}$.
A further optimization argument in $k$ implies $2^k\sim L^\frac{1}{k+1}$, which is nontrivial since $L>1$.
This finally results in the choice $k\sim\log(L)^\frac{1}{2}$.

%\emph{Step 2: Optimization.}
We let $u:=u_{\hat r,\hat k}$ and $\chi:=\chi_{\hat r,\hat k}$ with
$$
\hat k=c_1\log(L)^\frac{1}{2}
\quad \text{and} \quad
\hat r=c_2 L^\frac{\hat k}{\hat k+1}\sim L\exp\big(-C\log(L)^\frac{1}{2}\big),
$$
for some constants $C, c_1,c_2>0$.
We observe that condition \eqref{eq:non-deg} corresponds to \eqref{eq:geom-non-deg}
%\begin{equation}\label{eq:geom-non-deg}
%L\exp\big(-C\log(L)^\frac{1}{2}\big)<cH,
%\end{equation}
which is satisfied at least for all the inclusion domains such that $H\sim L$.
The result follows by noticing that, after optimization $\hat E(u,\chi)\sim H\hat r$.
\end{proof}

\begin{rmk}
We highlight that the condition \eqref{eq:geom-non-deg} can be viewed as a geometric information on the inclusion domains, in the sense that (scaling) optimal realizations of the type discussed above cannot be too thin.
\end{rmk}

\subsection{Lower bounds}

We complement the upper bound construction with an ansatz-free lower bound estimate.

\begin{prop}\label{prop:T4-lb}
Let $\hat{E}(V)$ be as in \eqref{eq:en-vol-scaled} and let $K$ be as in \eqref{eq:T4_wells}.
Then for every $V>1$ there holds
\begin{align*}
\hat{E}(V) \gtrsim V\exp\big(-C^{(1)}\log(V)^\frac{1}{2}\big),
\end{align*}
for some constant $C^{(1)}>0$.
\end{prop}

\begin{proof}
As in \cite{RT21}, we first observe that since the $\chi_{1,1}$ component determines the $\chi_{2,2}$ component and vice versa by \eqref{eq:nonlinearT4}, we can iterate the commutator bounds of Lemma \ref{lem:commutator2}. 
This yields that for any $m\in2\N$ it holds that
\begin{align*}
V &\le \|(\chi_{1,\mu,\mu_m}(D)+\chi_{2,\mu,\mu_{m+1}}(D))\chi\|_{L^2}^2+\|\chi_{1,\mu,\mu_m}(D)\chi_{1,1}-\chi_{1,1}\|_{L^2}^2 \\
& \quad +\|\chi_{2,\mu,\mu_{m+1}}(D)\chi_{2,2}-\chi_{2,2}\|_{L^2}^2 \\
&\lesssim \mu\mu_m^2V^2+C^m(\mu^{-2}+\mu_2^{-1})\hat E(\chi).
\end{align*}
Here we have already used that, by the assumption that $V\geq 1$ and by the isoperimetric estimates, we do not have losses due to the fact that $\hat E(\chi)\gtrsim V^\frac{1}{2}$. In particular, we may consider the commutator with some $\gamma$ fixed and still the leading term is the linear one.
Since $\mu_m\sim c^m\mu^{m-2}\mu_2$ for some $c>1$, we have
$$
V \lesssim c^{2m}\mu^{2m-3}\mu_2^2V^2+C^m(\mu^{-2}+\mu_2^{-1})\hat E(\chi).
$$
From $V\sim c^m\mu^{2m-3}\mu_2^2V^2$ we infer $\mu_2\sim V^{-\frac{1}{2}}c^{-m}\mu^{-\frac{2m-3}{2}}$ which gives
$$
V\lesssim C_0^m(\mu^{-2}+\mu^\frac{2m-3}{2}V^\frac{1}{2})\hat E(\chi)
$$
with $C_0>Cc$.
An optimization, i.e. $\mu\sim V^\frac{1}{2m+1}$, leads to
$$
V\lesssim C_0^m V^\frac{2}{2m+1}\hat E(\chi).
$$
By taking $m\in2\N$ such that $C_0^m\sim V^\frac{2}{2m+1}$, which is $m\sim\log(V)^\frac{1}{2}$, we obtain
$$
V \exp(-C^{(1)}\log(V)^\frac{1}{2})\lesssim\hat E(\chi).
$$
\end{proof}

\section*{Acknowledgements}
Both authors gratefully acknowledge funding by the Deutsche Forschungsgemeinschaft (DFG, German Research Foundation) through SPP 2256, project ID 441068247. 
A.R. is a member of the Heidelberg STRUCTURES Excellence Cluster, which is funded by the Deutsche Forschungsgemeinschaft (DFG, German Research Foundation) under Germany's Excellence Strategy EXC 2181/1 - 390900948.

\appendix

\section{Branching constructions in rectangular domains}

In this section we present the branching constructions as alternatives to the diamond-shaped constructions from the main body of the text.

\subsection{A branching construction for the upper bound from Theorem \ref{thm:2well_norm}}
\label{app:first}

We begin by presenting the branching construction for the upper bound from Theorem \ref{thm:2well_norm}.
To do so we refer to \cite[Lemma 3.2]{RT22} which in turn reworks \cite[Lemma 2.3]{CC15}:

\begin{prop}\label{prop:2wells-ub_branching}
Let $\hat{E}(V)$ be as in \eqref{eq:en-vol-scaled} and let $K$ be as in \eqref{eq:two_wells_proof}.
Let $V>1$ be given and let $\Omega=[0,L]\times[0,H]$ with $H>L>1$ and $HL=V$.
Then there exist $u\in W_0^{1,\infty}(\Omega;\R^2)$ and $\chi\in BV(\R^2;K_0)$ with $\supp(\chi)=\Omega$ such that
$$
\hat E(u,\chi)\lesssim V^\frac{3}{5}.
$$
\end{prop}

\begin{proof}
The construction is that of \cite[Lemma 3.2]{RT22} corresponding to $N=\lceil\frac{4L}{H}\rceil$, where $\lceil\cdot\rceil$ denotes the upper integer-part.
From formula (19) of \cite[Lemma 3.2]{RT22} with $p=2$ and $\epsilon=1$ we get
\begin{equation}\label{eq:bra1-ub}
\hat E(u,\chi) \le C\Big(\frac{L^3}{H}+H\Big),
\end{equation}
where $\chi$ is the projection of $\nabla u$ onto $K_0$ and $C>0$ is a positive constant depending on $K$.
Notice also that $1\le N\le 4$ and thus its contribution in (19) of \cite[Lemma 3.2]{RT22} is absorbed by the constant $C$ in the inequality above.
By an optimization argument we obtain $L\sim H^\frac{2}{3}$ which implies $V=H^\frac{5}{3}$ and the result follows.
\end{proof}

\subsection{A branching construction for the upper bound from Corollary \ref{cor:lamination_1}}
We complement the higher dimensional lens-shaped upper bound construction from the proof of Corollary \ref{cor:lamination_1} %Theorem \ref{thm:second} with a double branching construction in a thin rectangular domain.
with the corresponding branching construction.

\begin{proof}
We follow a similar construction as in the first order outer laminate from \cite[Proposition 6.3]{RT21}.
Let $u,\chi:\R^2\to\R^2$ be the functions introduced in the proof of Proposition \ref{prop:2wells-ub_branching} in the previous subsection.
Then, we consider $\Omega=[0,L]\times[0,H]^{n-1}$ and define $\tilde u(x_1,x_2,\dots,x_n)=u(x_1,\rho(x_2,\dots,x_n))$ and
$\tilde \chi(x_1,x_2,\dots,x_n)=\chi(x_1,\rho(x_2,\dots,x_n))$ with
$$
\rho(x_2,\dots,x_n)=\max_{2\le j\le n-1}\Big\{\Big|x_j-\frac{H}{2}\Big|\Big\}+\frac{H}{2}.
$$
Hence \eqref{eq:bra1-ub} gives
$$
\hat E(\tilde u,\tilde\chi)\lesssim \Big(\frac{L^3}{H}+H\Big)H^{n-2}.
$$
By optimization $L\sim H^\frac{2}{3}$, thus the constraint $LH^{n-1}=V$ yields $H\sim V^\frac{3}{3n-1}$ and thus the desired result.
\end{proof}

\subsection{A branching construction for the upper bound from Theorem \ref{thm:second}}
\label{app:second}

In addition to the lens-shaped construction from the proof of Theorem \ref{thm:second} which was presented in the main body of the text, an alternative upper bound construction for Theorem \ref{thm:second} can be realized by a double branching construction, where the coarsest oscillation happens a number of times of order $1$.
The construction retraces the steps of \cite[Section 5]{RT22} keeping track of $H$ and $L$ (that are constants there).

\begin{proof}[Proof of Proposition \ref{prop:3wells-ub} by means of a double branching construction]
As a first step we consider the construction given in Proposition \ref{prop:2wells-ub} with gradients
$$
B_1=\frac{1}{3}A_1+\frac{2}{3}A_2
\quad \text{and} \quad
B_2=\frac{1}{3}A_3+\frac{2}{3}A_4.
$$
We denote this macroscopic construction with $\tilde u$ and $\tilde\chi$.
Without loss of generality we assume that we have only one tree of branching.

We briefly describe the construction for the sake of clarity of exposition:
$\Omega$ is subdivided in $2^j$ identical cells equal (up to horizontal translation) to $[0,\ell_j]\times[h_j,h_{j+1}]$ for every $j\ge0$ corresponding to the generation.
This process stops when $j$ is such that $\ell_j>h_j$, then using a cut-off argument to match the austenite phase.
Here $\ell_j=L2^{-j}$ and $h_j=\frac{H}{2}\theta^j$ for some $\theta\in(\frac{1}{4},\frac{1}{2})$.
In each cell $\tilde u$ is defined by \cite[Lemma 3.1]{RT22} with $A=B_1$ and $B=B_2$.

Now we build $u$ and $\chi$ replacing $\tilde u$ and $\tilde\chi$ with a branching construction, having gradients $A_1$, $A_2$ on $\{\tilde\chi=B_1\}$ and $A_3$, $A_4$ on $\{\tilde\chi=B_2\}$, respectively.
For each generation the set $\{\tilde\chi=B_1\}$ consists (up to translation) of $R_j\cup\varphi(R_j)$ where
$$
R_j=\Big[0,\frac{\ell_j}{4}\Big]\times[0,h_j]
\quad \text{and} \quad
\varphi(x_1,x_2)=\Big(\frac{\ell_j x_2}{4h_j},x_2\Big).
$$
In each rectangle $R_j$ we apply \cite[Lemma 3.2]{RT22} with $p=2$ and $\epsilon=1$ (inverting the roles of $x_1$ and $x_2$); its energy contribution is
$$
(2\theta^3)^j\frac{H^3}{N_j^2 L}+2^{-j}LN_j.
$$
The shear gives an additional elastic energy term of order $(\frac{\ell_j}{h_j})^2h_j\ell_j$ thus, the contribution in $R_j\cup\varphi(R_j)$ is
$$
(2\theta^3)^j\frac{H^3}{N_j^2 L}+2^{-j}LN_j+\frac{\ell_j^3}{h_j}.
$$
We refer the interested reader to \cite[Section 5.1]{RT22} for details, in particular the energy contribution above corresponds to the first formula at page 28 there.
An optimization argument gives $N_j\sim HL^{-\frac{2}{3}}(\theta 2^\frac{2}{3})^j$.
The construction on $\{\tilde\chi=B_2\}$ is completely analogous.
Multiplying by the number of cells $2^j$ and summing over $j$ we infer
$$
\hat E(u,\chi) \lesssim \sum_{j=0}^\infty (\theta 2^\frac{2}{3})^j(HL^\frac{1}{3})+(\theta^{-1}2^{-3})^j \frac{L^3}{H}\lesssim HL^\frac{1}{3}+\frac{L^3}{H}.
$$
By another optimization argument we choose $L\sim H^\frac{3}{4}$ which implies (from the volume constraint) that $L\sim V^\frac{3}{7}, H\sim V^\frac{4}{7}$ and the result follows.
\end{proof}

\bibliographystyle{alpha}
\bibliography{citations1}

\end{document}